\documentclass[twoside]{article}

\newcommand{\heuteIst}{November 24, 2011, by TS}

\usepackage{amsmath,amsthm,amsfonts,latexsym,amscd,amssymb,enumerate}
\usepackage[hypertex]{hyperref}
\usepackage{color}

\theoremstyle{plain}
\newtheorem{theorem}{Theorem}
\newtheorem{lemma}{Lemma}
\newtheorem{corollary}{Corollary}
\newtheorem{proposition}{Proposition}

\theoremstyle{definition}
\newtheorem{claim}{Claim}
\newtheorem{definition}{Definition}
\newtheorem{example}{Example}

\newtheorem{remark}{Remark}
\newtheorem{question}{Question}
\newtheorem{problem}{Problem}

\newcommand{\naturals}{\mathbb{N}}
\newcommand{\reals}{\mathbb{R}}
\newcommand{\abs}[1]{|#1|}

\newcommand\Img{\operatorname{Im}}
\newcommand\Ker{\operatorname{Ker}}
\newcommand\Harm{\operatorname{Harm}}
\newcommand\diam{\operatorname{diam}}
\newcommand\sign{\operatorname{sign}}
\newcommand{\im}{\operatorname{im}}

\begin{document}
\pagestyle{myheadings}
\markboth{Laurent Bartholdi, Thomas Schick, Nat Smale, Steve Smale}{Hodge Theory on Metric
Spaces}

\date{Last compiled \today; last edited \heuteIst}

\title{Hodge Theory on Metric Spaces}

\author{Laurent Bartholdi\footnotemark[1]\\ Georg-August-Universit\"at G{\"o}ttingen\\
G\"ottingen, Germany\and Thomas Schick\thanks{
email: \protect\href{mailto:schick@uni-math.gwdg.de}{\texttt{laurent@uni-math.gwdg.de}}
and \protect\href{mailto:schick@uni-math.gwdg.de}{\texttt{schick@uni-math.gwdg.de}}
\protect\\
www: \protect\href{http://www.uni-math.gwdg.de/schick}{\texttt{http://www.uni-math.gwdg.de/schick}}
\protect\\
Laurent Bartholdi and Thomas Schick were partially supported by the Courant
Research Center ``Higher order structures in Mathematics'' of the German
Initiative of Excellence
}\\
Georg-August-Universit\"at G{\"o}ttingen\\
G\"ottingen, Germany \and Nat Smale\\ University of Utah\\ Salt Lake City, USA \and Steve Smale\thanks{Steve Smale was supported in part by the NSF and the Toyota
Technological Institute, Chicago}\\  City University of Hong Kong\\ Pokfulam,
P.R. China
\and \textit{With an appendix by Anthony W. Baker}\thanks{email: \texttt{anthony.w.baker@boeing.com}}\\  \textit{Mathematics and Computing Technology}\\
  \textit{The Boeing Company}\\ \textit{Chicago, USA}.
}

\maketitle

\begin{abstract}
  Hodge theory is a beautiful synthesis of geometry, topology, and
analysis, which has been developed in the setting of Riemannian manifolds. On
the other hand, spaces of images, which are important in the mathematical
foundations of vision and pattern recognition, do not fit this framework. This
motivates us to develop a version of Hodge theory on metric spaces with a
probability measure. We believe that this constitutes a step towards
understanding the geometry of vision.

The appendix by Anthony Baker provides a separable, compact metric space with
infinite dimensional $\alpha$-scale homology.
\end{abstract}

\section{Introduction}\label{Section1}

Hodge Theory \cite{13} studies the relationships of topology, functional
analysis and geometry of a manifold. It extends the theory of the Laplacian on
domains of Euclidean space or on a manifold.

However, there are a number of spaces, not manifolds, which could benefit from
an extension of Hodge theory, and that is the motivation here. In particular
we believe that a deeper analysis in the theory of vision could be led by
developments of Hodge type. Spaces of images are important for developing a
mathematics of vision (see e.g.~Smale, Rosasco, Bouvrie, Caponnetto, and
Poggio \cite{20}); but these spaces are far from possessing manifold
structures. Other settings include spaces occurring in quantum field theory,
manifolds with singularities and/or non-uniform measures.

A number of previous papers have given us inspiration and guidance. For
example there are those in combinatorial Hodge theory of Eckmann \cite{8},
Dodziuk \cite{7}, Friedman \cite{11}, and more recently Jiang, Lim, Yao,
and Ye \cite{16}. Recent decades have seen extensions of the Laplacian from
its classical setting to that of combinatorial graph theory. See e.g.~Fan
Chung \cite{5}. Robin Forman \cite{10} has useful extensions from
manifolds. Further extensions and relationships to the classical settings are
Belkin, Niyogi \cite{2}, Belkin, De Vito, and Rosasco \cite{1}, Coifman,
Maggioni \cite{6}, and Smale, Zhou \cite{19}.

Our approach starts with a metric space $X$ (complete, separable), endowed
with a probability measure. For $\ell\geq 0$, an $\ell$-form is a function on
$(\ell+1)$-tuples of points in $X$. The coboundary operator $\delta$ is defined
from $\ell$-forms to $(\ell+1)$-forms in the classical way following \v Cech,
Alexander, and Spanier. Using the $L^2$-adjoint $\delta^{\ast}$ of $\delta$
for a boundary operator, the $\ell$th order Hodge operator on $\ell$-forms is
defined by $\Delta_{\ell}=\delta^{\ast}\delta+\delta\delta^{\ast}$. The
harmonic $\ell$-forms on $X$ are solutions of the equation
$\Delta_{\ell}(f)=0$. The $\ell$-harmonic forms reflect the $\ell$th homology
of $X$ but have geometric features. The harmonic form is a special
representative of the homology class and it may be interpreted as one
satisfying an optimality condition. Moreover, the Hodge equation is linear and
by choosing a finite sample from $X$ one can obtain an approximation of this
representative by a linear equation in finite dimension.

There are two avenues to develop this Hodge theory. The first is a kernel
version corresponding to a Gaussian or a reproducing kernel Hilbert
space. Here the topology is trivial but the analysis gives a substantial
picture. The second version is akin to the adjacency matrix of graph theory
and corresponds to a threshold at a given scale $\alpha$. When $X$ is finite
this picture overlaps with that of the combinatorial Hodge theory referred to
above.

For passage to a continuous Hodge theory, one encounters:

\begin{problem}[Poisson Regularity Problem] If $\Delta_{\ell}(f)=g$ is
  continuous, under what
  conditions is $f$ continuous?
\end{problem}
It is proved that a positive solution of the Poisson Regularity Problem
implies a complete Hodge decomposition for continuous $\ell$-forms in the
``adjacency matrix'' setting (at any scale $\alpha$), provided the $L^2$-cohomology
is finite dimensional. The problem is solved affirmatively for some cases as
$\ell=0$, or $X$ is finite. One special case is

\begin{problem} Under what conditions are harmonic $\ell$-forms continuous?
\end{problem}
Here we have a solution for $\ell=0$ and $\ell=1$.

The solution of these regularity problems would be progress toward the
important cohomology identification problem: To what extent does the $L^2$-cohomology coincide with the classical cohomology?
 We have an answer to this question, as well as a full Hodge theory in the special, but important case of Riemannian manifolds. The following theorem is proved in Section \ref{Section9} of this paper.

\begin{theorem} Suppose that $M$ is a compact Riemannian manifold, with strong convexity radius $r$ and that $k>0$ is an upper bound on the sectional curvatures.
Then, if $0<\alpha<\text{max}\{r,\sqrt{\pi}/2k\}$, our Hodge theory
holds. That is, we have a Hodge decomposition, the kernel of
$\Delta_{\ell}$ is isomorphic to the $L^2$-cohomology, and to
the de Rham cohomology of $M$ in degree $\ell$.
\end{theorem}

More general conditions on a metric space $X$ are given in Section
\ref{Section9}.

Certain previous studies show how topology questions can give insight into the
study of images. Lee, Pedersen, and Mumford \cite{15} have investigated
$3\times 3$ pixel images from real world data bases to find the evidence for
the occurrence of  homology classes of degree $1$. Moreover, Carlsson,
Ishkhanov, de Silva, and Zomorodian \cite{3} have found evidence for homology
of surfaces in the same data base. Here we are making an attempt to give some
foundations to these studies. Moreover, this general Hodge theory could yield
optimal representatives of the homology classes and provide systematic
algorithms.

Note that the problem of recognizing a surface is quite complex; in
particular, the cohomology of a non-oriented surface has torsion, and
it may seem naive to attempt to recover such information from
computations over $\mathbb R$. Nevertheless, we shall argue that Hodge
theory provides a rich set of tools for object recognition, going
strictly beyond ordinary real cohomology.

Related in spirit to our $L^2$-cohomology, but in a quite different setting,
is the $L^2$-cohomology as introduced by Atiyah \cite{Atiyah}. This is
defined either via $L^2$-differential forms \cite{Atiyah} or combinatorially \cite{Dodziuk(1977)}, but again with
an $L^2$ condition. Questions like the Hodge decomposition problem also arise
in this setting, and its failure gives rise to additional invariants, the
Novikov-Shubin invariants. This theory has been extensively studied, compare
e.g.~\cite{CheegerGromov,Luck2,Schick1,LinnellSchick} for important properties
and geometric as well as algebraic applications. In
\cite{Luck1,MR1828605,MR1990479} approximation of the $L^2$-Betti numbers for
infinite simplicial complexes in terms of associated finite simplicial
complexes is discussed in increasing generality. Complete calculations of the
spectrum of the associated Laplacian are rarely possible, but compare
\cite{DicksSchick} for one of these cases%
. 
The monograph \cite{LuckBuch} provides rather complete information about this
theory. Of particular relevance for the present paper is Pansu's
\cite{MR1377309} where in Section 4 he introduces an $L^2$-Alexander-Spanier complex
similar to ours. He uses it to prove homotopy invariance of $L^2$-cohomology
---that way identifying its cohomology with $L^2$-de Rham cohomology and
$L^2$-simplicial cohomology (under suitable assumptions).

Here is some background to the writing of this paper. Essentially
Sections \ref{Section2} through \ref{Section8} were in a finished paper by Nat Smale and Steve Smale, February 20,
2009. That version stated that the coboundary operator of Theorem
\ref{Theorem3}, Section \ref{Section4}
must have a closed image. Thomas Schick pointed out that this assertion was
wrong, and in fact produced a counterexample, now Section \ref{Section10} of
this
paper. Moreover, Schick and Laurent Bartholdi set in motion the proofs that
give the sufficient conditions for the finite dimensionality of the $L^2$-cohomology groups in Section \ref{Section9} of this paper, and hence the
property that the image of the
coboundary is closed. In particular Theorems \ref{Theorem9.1} and
\ref{Theorem9.2} were proved by them.

Some conversations with Shmuel Weinberger were helpful.

\section[An L2-Hodge Theory]{\protect\boldmath An $L^2$-Hodge Theory}\label{Section2}

In this section we construct a general Hodge Theory for certain
$L^2$-spaces over $X$, making only use of a probability measure on a
set $X$.

As to be expected, our main result (Theorem~\ref{Theorem1}) shows that
homology is trivial under these general assumptions. This is a
backbone for our subsequent elaborations, in which a metric will be
taken into account to obtain non-trivial homology.

This is akin to the construction of Alexander-Spanier cohomology in
topology, in which a chain complex with trivial homology (which does
not see the space's topology) is used to manufacter the standard
Alexander-Spanier complex.

The amount of structure needed for our theory is minimal. First, let
us introduce some notation used throughout the section.  $X$ will
denote a set endowed with a probability measure $\mu$
($\mu(X)=1$). The $\ell$-fold cartesian product of $X$ will be denoted
as $X^{\ell}$ and $\mu_{\ell}$ will denote the product measure on
$X^{\ell}$. A useful example to keep in mind is: $X$ a compact domain
in Euclidean space, $\mu$ the normalized Lebesgue measure. More
generally, one may take $\mu$ a Borel measure, which need not be the
Euclidean measure.

Furthermore, we will assume the existence of a kernel function
$K\colon X^2 \to\mathbb{R}$, a non-negative, measurable, symmetric
function which we will assume is in $L^{\infty}(X\times X)$, and for
certain results, we will impose additional assumptions on $K$.

We may consider, for simplicity, the constant kernel $K\equiv 1$; but
most proofs, in this section, cover with no difficulty the general
case, so we do not impose yet any restriction to $K$. Later sections,
on the other hand, will concentrate on $K\equiv1$, which already
provides a very rich theory.

The kernel $K$ may be used to conveniently encode the notion of
locality in our probability space $X$, for instance by defining it as
the Gaussian kernel $K(x,y)=e^{-\frac{\|x-y\|^2}{\sigma}}$,
for some $\sigma>0$.

Recall that a chain complex of vector spaces is a sequence of vector spaces
$V_j$ and linear maps $d_j\colon V_j\to V_{j-1}$ such that the composition
$d_{j-1}\circ d_j=0$. A co-chain complex is the same, except that $d_j\colon
V_j\to 
V_{j+1}$. The basic spaces in this section are $L^2(X^{\ell})$, from which we
will construct chain and cochain complexes:
\begin{equation}\label{2.1}
\begin{CD}
  \cdots @>{\partial_{\ell+1}} >> L^2(X^{\ell+1}) @> {\partial_{\ell}} >>
  L^2(X^{\ell}) @> {\partial_{\ell -1}} >> \cdots L^2(X) @>{ \partial_0} >> 0
\end{CD}
\end{equation}
and
\begin{equation}\label{2.2}
  \begin{CD}
    0 @>>> L^2(X) @>{\delta_0} >> L^2(X^2) @>{\delta_1} >> \cdots
    @>{\delta_{\ell-1} }>> L^2(X^{\ell+1}) @>{\delta_{\ell}} >> \cdots
  \end{CD}
\end{equation}

Here, both $\partial_{\ell}$ and $\delta_{\ell}$ will be bounded linear maps,
satisfying $\partial_{\ell-1}\circ \partial_{\ell}=0$ and $\delta_{\ell}\circ
\delta_{\ell-1}=0$. When there is no confusion, we will omit the subscripts of
these operators.

We first define $\delta=\delta_{\ell-1}\colon L^2(X^{\ell})\to L^2(X^{\ell+1})$ by
\begin{equation}\label{2.3}
  \begin{CD}
    \delta f(x_0,\dots,x_{\ell})=\sum_{i=0}^{\ell} (-1)^i \prod_{j\neq
      i}\sqrt{K(x_i,x_j)}f(x_0,\dots,\hat x_i,\dots,x_{\ell})
  \end{CD}
\end{equation}
where $\hat x_i$ means that $x_i$ is deleted. This is similar to the
co-boundary operator of Alexander-Spanier cohomology (see Spanier
\cite{21}). The square root in the formula is unimportant for most of the
sequel, and is there so that when we define the Laplacian on $L^2(X)$, we
recover the operator as defined in Gilboa and Osher \cite{12}. We also note
that in the case $X$ is a finite set, $\delta_0$ is essentially the same as
the gradient operator developed by Zhou and Sch\"olkopf \cite{24} in the
context of learning theory.

\begin{proposition}\label{Proposition1}
  For all $\ell\geq 0$, $\delta\colon L^2(X^{\ell})\to
  L^2(X^{\ell+1})$ is a bounded linear map.
\end{proposition}
\begin{proof}
 Clearly $\delta f$ is measurable, as $K$ is measurable. Since
  $\|K \|_{\infty}<\infty$, it follows from the Schwartz inequality in
  $\mathbb{R}^\ell$ that
\begin{align*}
|\delta f(x_0,\dots,x_{\ell})|^2 &\leq \|K\|^{\ell}_{\infty}\left (\sum_{i=0}^{\ell}|f(x_0,\dots,\hat x_i,\dots,x_{\ell})|\right )^2\\
&\leq \|K\|^{\ell}_{\infty}(\ell +1)\sum_{i=0}^{\ell}|f(x_0,\dots,\hat x_i,\dots,x_{\ell})|^2.
\end{align*}
Now, integrating both sides of the inequality
with respect to $d\mu_{\ell+1}$ , using Fubini's Theorem on the right side and
the fact that $\mu(X)=1$ gives us
$$
\|\delta f\|_{L^2(X^{\ell+1})}\leq \sqrt{\|K\|^{\ell}_{\infty}}(\ell+1) \|f\|_{L^2(X^{\ell})},
$$
completing the proof.
\end{proof}

  Essentially the same proof shows that $\delta$ is a bounded linear map on
  $L^p$, $p\geq 1$.
  \begin{proposition}\label{Proposition2}
    For all $\ell\geq 1$, $\delta_{\ell}\circ\delta_{\ell-1}=0$.
  \end{proposition}
\begin{proof}
  The proof is standard when $K\equiv 1$. For $f\in L^2(X^{\ell})$ we have
\begin{align*}
&\delta_{\ell}(\delta_{\ell-1} f)(x_0,\dots,x_{\ell+1})\\
&=\sum_{i=0}^{\ell+1}(-1)^i\prod_{j\neq i}\sqrt{K(x_i,x_j)}(\delta_{\ell-1} f)(x_0,\dots,\hat x_i,\dots,x_{\ell+1})\\
&=\sum_{i=0}^{\ell+1}(-1)^i\prod_{j\neq i}\sqrt{K(x_i,x_j)}\sum_{k=0}^{i-1}(-1)^k\prod_{n\neq k,i}\sqrt{K(x_k,x_n)}f(x_0,\dots,\hat x_k,\dots,\hat x_i,\dots,x_{\ell+1})\\
&+\sum_{i=0}^{\ell+1}(-1)^i\prod_{j\neq
  i}\sqrt{K(x_i,x_j)}\sum_{k=i+1}^{\ell+1}(-1)^{k-1}\prod_{n\neq
  k,i}\sqrt{K(x_k,x_n)}f(x_0,\dots,\hat x_i,\dots,\hat x_k,\dots,x_{\ell+1})
\end{align*}

Now we note that on the right side of the second equality for given $i,k$ with
$k<i$, the corresponding term in the first sum
$$(-1)^{i+k}\prod_{j\neq i}\sqrt{K(x_i,x_j)}\prod_{n\neq k,i}\sqrt{K(x_k,x_n)} f(x_0.\dots,\hat x_k,\dots,\hat x_i,\dots,x_{\ell+1})
$$
cancels the term in the second sum where $i$ and $k$ are reversed
$$
(-1)^{k+i-1}\prod_{j\neq k}\sqrt{K(x_k,x_j)}\prod_{n\neq
  k,i}\sqrt{K(x_k,x_n)} f(x_0.\dots,\hat x_k,\dots,\hat x_i,\dots,x_{\ell+1})
$$
because, using the symmetry of $K$,
\[
\prod_{j\neq i}\sqrt{K(x_i,x_j)}\prod_{n\neq
  k,i}\sqrt{K(x_k,x_n)}=\prod_{j\neq k}\sqrt{K(x_k,x_j)}\prod_{n\neq
  k,i}\sqrt{K(x_i,x_n)}.\qedhere
\]
\end{proof}
It follows that \eqref{2.2} and \eqref{2.3} define a co-chain complex. We now
define, for
$\ell>0$, $\partial_{\ell}\colon L^2(X^{\ell+1})\to L^2(X^{\ell})$ by
\begin{equation}\label{2.4}
\partial_{\ell}g(x)=\sum_{i=0}^{\ell}(-1)^i\int_X\left
  (\prod_{j=0}^{\ell-1}\sqrt{K(t,x_j)} \right
)g(x_0,\dots,x_{i-1},t,x_i,\dots,x_{\ell-1})\,d\mu(t)
\end{equation}
where $x=(x_0,\dots,x_{\ell-1})$ and for $\ell=0$ we define
$\partial_0\colon L^2(X)\to 0$.

\begin{proposition}\label{Proposition3}
 For all $\ell\geq 0$,
  $\partial_{\ell}\colon L^2(X^{\ell+1})\to L^2(X^{\ell})$ is a bounded linear map.
\end{proposition}
\begin{proof}
  For $g\in L^2(X^{\ell+1})$, we have
\begin{align*}
|\partial_{\ell}g(x_0,\dots,x_{\ell-1})| &\leq \|K\|_{\infty}^{(\ell-1)/2}\sum_{i=0}^{\ell}\int_X |g(x_0,\dots,x_{i-1},t,\dots,x_{\ell-1})|\,d\mu(t)\\
&\leq \|K\|_{\infty}^{(\ell-1)/2}\sum_{i=0}^{\ell}\left (\int_X|g(x_0,\dots,x_{i-1},t,\dots,x_{\ell-1})|^2\,d\mu(t) \right )^{\frac{1}{2}}\\
&\leq \|K\|_{\infty}^{(\ell-1)/2}\sqrt{\ell+1} \left
  (\sum_{i=0}^{\ell}\int_X|g(x_0,\dots,x_{i-1},t,\dots,x_{\ell-1})|^2\,d\mu(t)
\right )^{\frac{1}{2}}
\end{align*}
where we have used the Schwartz inequalities for $L^2(X)$ and $\mathbb{R}^{\ell+1}$ in the second and third inequalities respectively. Now, square
both sides of the inequality and integrate over $X^{\ell}$ with respect to
$\mu_{\ell}$ and use Fubini's Theorem arriving at the following bound to
finish the proof:
\[
\|\partial_{\ell}g\|_{L^2(X^{\ell})}\leq
\|K\|_{\infty}^{(\ell-1)/2}(\ell+1)\|g\|_{L^2(X^{\ell+1})}.\qedhere
\]
\end{proof}
\begin{remark}
   As in Proposition \ref{Proposition1}, we can replace $L^2$ by $L^p$, for
  $p\geq 1$.
\end{remark}

We now show that (for $p=2$) $\partial_{\ell}$ is actually the adjoint of
$\delta_{\ell-1}$ (which gives a second proof of Proposition
\ref{Proposition3}).
\begin{proposition}\label{Proposition4}
   $\delta_{\ell-1}^{\ast}=\partial_{\ell}$. That is
  $\langle \delta_{\ell-1}f,g\rangle_{L^2(X^{\ell+1})}=\langle f,\partial_{\ell}g\rangle_{L^2(X^{\ell})}$
  for all $f\in L^2(X^{\ell})$ and $g\in L^2(X^{\ell+1})$.
  \end{proposition}
  \begin{proof}
    For $f\in L^2(X^{\ell})$ and $g\in L^2(X^{\ell+1})$ we have, by Fubini's
    Theorem
\begin{equation*}
  \begin{split}
    \langle\delta_{\ell-1}f,g\rangle =&\sum_{i=0}^{\ell}(-1)^i
    \int_{X^{\ell+1}} \prod_{j\neq i}\sqrt{K(x_i,x_j)}f(x_0,\dots,\hat
    x_i,\dots,x_{\ell})g(x_0,\dots,x_{\ell})\, d\mu_{\ell+1}\\ 
    =&\sum_{i=0}^{\ell}(-1)^i\int_{X^{\ell}} f(x_0,\dots,\hat
    x_i,\dots,x_{\ell})\cdot\\
    & \cdot \int_X\prod_{j\neq
      i}\sqrt{K(x_i,x_j)}g(x_0,\dots,x_{\ell})\,d\mu(x_i)\,d\mu(x_0)\cdots
    \widehat{d\mu(x_i)}\cdots d\mu(x_{\ell})
  \end{split}
\end{equation*}

In the $i$-th term on the right, relabeling the variables $x_0,\dots,\hat
x_i,\dots x_{\ell}$ with $y=(y_0,\dots,y_{\ell-1})$ (that is
$y_j=x_{j+1}$ for $j\geq i$) and putting the sum inside the integral gives us
$$
\int_{X^{\ell}}f(y)\sum_{i=0}^{\ell}(-1)^i\int_X\prod_{j=0}^{\ell-1}\sqrt{K(x_i,y_j)}g(y_0,\dots,y_{i-1},x_i,y_i,\dots,y_{\ell-1})\,
d\mu(x_i)\, d\mu_{\ell}(y)
$$
which is just $\langle f,\partial_{\ell}g\rangle$.
\end{proof}

We note, as a corollary, that $\partial_{\ell-1}\circ \partial_{\ell}=0$, and
thus \eqref{2.1} and \eqref{2.4} define a chain complex. We can thus define the homology
and cohomology spaces (real coefficients) of \eqref{2.1} and \eqref{2.2} as follows. Since
$\Img\partial_{\ell}\subset \Ker\partial_{\ell-1}$ and
$\Img\delta_{\ell-1}\subset\Ker\delta_{\ell}$ we define the
quotient spaces
\begin{equation}\label{2.5}
H_{\ell}(X)=H_{\ell}(X,K,\mu)=\frac{\Ker\partial_{\ell-1}}{\Img\partial_{\ell}}\
\ \ \ \
H^{\ell}(X)=H^{\ell}(X,K,\mu)=\frac{\Ker\delta_{\ell}}{\Img\delta_{\ell-1}}
\end{equation}
which will be referred to the $L^2$-homology and cohomology
of degree $\ell$, respectively. In later sections, with additional assumptions on $X$ and $K$,
we will investigate the relation between these spaces and the topology of $X$,
for example, the Alexander-Spanier cohomology. In order to proceed with the
Hodge Theory, we consider $\delta$ to be the analogue of the exterior
derivative $d$ on $\ell$-forms from differential topology, and
$\partial=\delta^{\ast}$ as the analogue of $d^{\ast}$. We then define the
Laplacian (in analogy with the Hodge Laplacian) to be $\Delta_{\ell} =
\delta_{\ell}^{\ast}\delta_{\ell}+\delta_{\ell-1}\delta_{\ell-1}^{\ast}$. Clearly
$\Delta_{\ell}\colon  L^2(X^{\ell+1})\to L^2(X^{\ell+1})$ is a bounded, self
adjoint, positive semi-definite operator since for $f\in L^2(X^{\ell+1})$
\begin{equation}\label{2.6}
\langle\Delta f,f\rangle=\langle\delta^{\ast}\delta f,f\rangle+\langle\delta\delta^{\ast}f,f\rangle=\|\delta
f\|^2+\|\delta^{\ast}f\|^2
\end{equation}
where we have left off the subscripts on the operators. The Hodge Theorem will
give a decomposition of $L^2(X^{\ell+1})$ in terms of the image spaces under
$\delta$, $\delta^{\ast}$ and the kernel of $\Delta$, and also identify the
kernel of $\Delta$ with $H^{\ell}(X,K,\mu)$. Elements of the kernel of
$\Delta$ will be referred to as harmonic. For $\ell=0$, one easily computes
that
$$
\frac{1}{2}\Delta_0 f(x)=D(x)f(x)-\int_X K(x,y)f(y)\, d\mu(y)\ \ \
\text{where}\ \ D(x)=\int_X K(x,y)\, d\mu(y)
$$
which, in the case $K$ is a positive definite kernel on $X$, is the Laplacian
defined in Smale and Zhou \cite{19} (see section 5 below).

\begin{remark}
  It follows from \eqref{2.6} that $\Delta f=0$ if and only if
  $\delta_{\ell} f=0$ and $\delta^{\ast}_{\ell-1} f=0$, and so
  $\Ker\Delta_{\ell}= \Ker\delta_{\ell}\cap\Ker\delta_{\ell-1}^\ast$;
  in other words, a form is harmonic if and only if it is both closed
  and coclosed.
\end{remark}

\noindent The main goal of this section is the following $L^2$-Hodge theorem.

\begin{theorem}\label{Theorem1}
  Assume that $0<\sigma\leq K(x,y)\leq \|K\|_{\infty}<\infty$ almost
  everywhere. Then we have trivial $L^2$-cohomology in the
    following sense:
    \begin{equation*}
      \im(\delta_\ell)=\ker(\delta_{\ell+1})\qquad\forall\ell\ge 0.
    \end{equation*}
In particular, $H^{\ell}(X)=0$ for $\ell>0$ and we have by Lemma \ref{Lemma1} the (trivial)
orthogonal, direct sum decomposition 
$$
L^2(X^{\ell+1})=\Img\delta_{\ell-1}\oplus
\Img\delta_{\ell}^{\ast}\oplus \Ker\Delta_{\ell}
$$
and the cohomology space $H^{\ell}(X,K,\mu)$ is isomorphic to
$\Ker\Delta_{\ell}$, with each equivalence class in the former having a
unique representative in the latter.

For $\ell>0$, of course $\Ker\Delta_\ell=\{0\}$. For $\ell=0$,
$\Ker\Delta_0=\ker\delta_0\cong\mathbb{R}$ consists precisely of the constant functions.
\end{theorem}


In subsequent sections we will have occasion to use the $L^2$-spaces of
alternating functions:

\begin{align*}
L^2_a(X^{\ell+1})= &\{f\in L^2(X^{\ell+1})\colon f(x_0,\dots,x_{\ell})=(-1)^{\sign\sigma}f(x_{\sigma(0)},\dots,x_{\sigma(\ell)}), \\
& \sigma \ \text{a permutation}\}
\end{align*}

Due to the symmetry of $K$, it is easy to check that the coboundary $\delta$
preserves the alternating property, and thus Propositions \ref{Proposition1} through \ref{Proposition4}, as well
as formulas \eqref{2.1}, \eqref{2.2}, \eqref{2.5} and \eqref{2.6} hold with
$L^2_a$ in place of
$L^2$. We note that the alternating map
$$
\text{Alt}\colon L^2(X^{\ell+1})\to L^2_a(X^{\ell+1})
$$
defined by
$$
\text{Alt}(f)(x_0,\dots,x_{\ell}):=\frac{1}{(\ell+1)!}\sum_{\sigma\in
  S_{\ell+1}}(-1)^{\sign\sigma}f(x_{\sigma(0)},\dots,x_{\sigma(\ell)})
$$
is a projection relating the two definitions of $\ell$-forms. It is easy to
compute that this is actually an orthogonal projection, its inverse is just
the inclusion map.
\begin{remark}
  It follows from homological algebra that these maps induce
  inverse to each other isomorphisms of the cohomology groups we
  defined. Indeed, there is a standard chain homotopy between a variant of the
  projection 
  $\text{Alt}$ and the identity, givenq by
  $hf(x_0,\dots,x_n)=\frac{1}{n}\sum_{i=0}^nf(x_i,x_0,\dots,x_n)$. Because
  many formulas simplify, from now on 
  we will therefore most of the time work with the subcomplex of alternating
  functions.
\end{remark}

We first recall some relevant facts in a more abstract setting in the
following
\begin{lemma}[Hodge Lemma]\label{Lemma1}
  Suppose we have the cochain and corresponding dual chain
complexes
$$
\begin{CD}
  0 @>>> V_0 @>{\delta_0} >> V_1 @>{\delta_1} >> \cdots @>{\delta_{\ell-1} }>>
  V_{\ell} @>{\delta_{\ell}} >> \cdots
\end{CD}
$$
$$
\begin{CD}
\cdots @>{\delta_{\ell}^{\ast}} >> V_{\ell} @> {\delta_{\ell-1}^{\ast}} >>
V_{\ell-1} @> {\delta_{\ell -2}^{\ast}} >> \cdots @>{\delta_0^{\ast}} >>V_0
@>>> 0
\end{CD}
$$
where for $\ell=0,1,\dots$, $V_{\ell},\langle,\rangle_{\ell}$ is a Hilbert space,
$\delta_{\ell}$ (and thus $\delta_{\ell}^{\ast}$, the adjoint of
$\delta_{\ell}$) is a bounded linear map with $\delta^2=0$. Let
$\Delta_{\ell}=\delta_{\ell}^{\ast} \delta_{\ell}+\delta_{\ell-1}
\delta_{\ell-1}^{\ast}$. Then the following are equivalent:
\begin{enumerate}\item[(1)] $\delta_{\ell}\ \text{has closed range for all}\
  \ell$;
  \item[(2)] $\delta_{\ell}^{\ast}\ \text{has closed range for all}\ \ell$.
  \item[(3)] $\Delta_{\ell}=\delta^*_{\ell}\delta_\ell +
    \delta_{\ell-1}\delta_{\ell-1}^*$ has closed range for all $\ell$.
\end{enumerate}
Furthermore, if one of the above conditions hold, we have the orthogonal,
direct sum decomposition into closed subspaces
\begin{equation}\label{eq:spec_Hodge}
V_{\ell}=\Img\delta_{\ell-1}\oplus \Img\delta_{\ell}^{\ast}\oplus
\Ker\Delta_{\ell}
\end{equation}
and the quotient space
$\frac{\Ker\delta_{\ell}}{\Img\delta_{\ell-1}}$ is isomorphic to
$\Ker\Delta_{\ell}$, with each equivalence class in the former having a
unique representative in the latter.
\end{lemma}
\begin{proof}
  We first assume conditions (1) and (2) above and prove the decomposition. For
  all $f\in V_{\ell-1}$ and $g\in V_{\ell+1}$ we have
$$
\langle\delta_{\ell-1}f,\delta_{\ell}^{\ast}g\rangle_{\ell}=\langle\delta_{\ell}\delta_{\ell-1}f,g\rangle_{\ell+1}=0.
$$
Also, as in \eqref{2.6}, $\Delta_{\ell}f=0$ if and only if $\delta_{\ell}f=0$ and
$\delta_{\ell-1}^{\ast}f=0$.  Therefore, if $f\in\Ker\Delta_{\ell}$,
then for all $g\in V_{\ell-1}$ and $h\in V_{\ell+1}$
$$
\langle
f,\delta_{\ell-1}g\rangle_{\ell}=\langle\delta_{\ell-1}^{\ast}f,g\rangle_{\ell-1}=0
\ \ \ \
\text{and}\ \ \ \
\langle f,\delta_{\ell}^{\ast}h\rangle_{\ell}=\langle\delta_{\ell}f,h\rangle_{\ell+1}=0
$$
and thus $\Img\delta_{\ell-1}$, $\Img\delta_{\ell}^{\ast}$ and
$\Ker\Delta_{\ell}$ are mutually orthogonal. We now show that
$\Ker\Delta_{\ell}\supseteq(\Img\delta_{\ell-1}\oplus
\Img\delta_{\ell}^{\ast})^{\perp}$. This implies the orthogonal decomposition
\begin{equation}\label{eq:gen_Hodge}
  V_\ell=\ker(\Delta_{\ell})\oplus \overline{\Img(\delta_{\ell-1})}\oplus
    \overline{\Img(\delta_{\ell}^*)}.
\end{equation}
If (1) and (2) hold this implies the Hodge decomposition \eqref{eq:spec_Hodge}.
Let $v\in
(\Img\delta_{\ell-1}\oplus \Img\delta_{\ell}^{\ast})^{\perp}$. Then,
for all $w\in V_{\ell}$,
$$
\langle\delta_{\ell}v,w\rangle=\langle v,\delta_{\ell}^{\ast}w\rangle=0\qquad\text{and}\qquad
\langle \delta_{\ell-1}^{\ast}v,w\rangle=\langle v,\delta_{\ell-1}w\rangle=0,
$$
 which implies that
$\delta_{\ell}v=0$ and $\delta_{\ell-1}^{\ast}v=0$ and as noted above this
implies that $\Delta_{\ell}v=0$, proving the decomposition.

We define an isomorphism
$$
\tilde P \colon  \frac{\Ker\delta_{\ell}}{\Img\delta_{\ell-1}}\to
\Ker\Delta_{\ell}
$$
as follows. Let $P\colon  V_{\ell}\to \Ker\Delta_{\ell}$ be the orthogonal
projection. Then, for an equivalence class $[f]\in
\frac{\Ker\delta_{\ell}}{\Img\delta_{\ell-1}}$ define $\tilde
P([f])=P(f)$. Note that if $[f]=[g]$ then $f=g+h$ with
$h\in\Img\delta_{\ell-1}$, and therefore $P(f)-P(g)=P(h)=0$ by the
orthogonal decomposition, and so $\tilde P$ is well defined, and linear as $P$
is linear. If $\tilde P([f])=0$ then $P(f)=0$ and so $f\in
\Img\delta_{\ell-1}\oplus \Img\delta_{\ell}^{\ast}$. But
$f\in\Ker\delta_{\ell}$, and so, for all $g\in V_{\ell+1}$ we have
$\langle \delta_{\ell}^{\ast}g,f\rangle=\langle g,\delta_{\ell}f\rangle=0$, and thus $f\in
\Img\delta_{\ell-1}$ and therefore $[f]=0$ and $\tilde P$ is
injective. On the other hand, $\tilde P$ is surjective because, if
$w\in\Ker\Delta_{\ell}$, then $w\in\Ker\delta_{\ell}$ and so
$\tilde P([w])=P(w)=w$.

Finally, the equivalence of conditions (1), (2), {and (3) is a general fact about Hilbert
spaces and Hilbert cochain complexes}. If $\delta\colon V\to H$ is a bounded 
linear map between Hilbert spaces, and
$\delta^{\ast}$ is its adjoint, and if $\Img\delta$ is closed in $H$,
then $\Img\delta^{\ast}$ is closed in $V$. We include the proof for
completeness. Since $\Img\delta$ is closed, the bijective map
$$
\delta\colon (\Ker\delta)^{\perp}\to\Img\delta
$$
is an isomorphism by the open mapping theorem. It follows that the norm of $\delta^-1$,
$$
\inf\{\|\delta(v)\|\colon v\in(\Ker\delta)^{\perp}, \ \|v\|=1\}>0.
$$
Since $\Img\delta\subset(\Ker\delta^{\ast})^{\perp}$, it suffices
to show that
$$
\delta^{\ast}\delta\colon (\Ker\delta)^{\perp}\to(\Ker\delta)^{\perp}
$$
is an isomorphism, for then
$\Img\delta^{\ast}=(\Ker\delta)^{\perp}$ which is closed. However,
this is established by noting that $\langle \delta^{\ast}\delta v,v\rangle=\|\delta v\|^2$
and the above inequality imply that
$$
\inf\{\langle \delta^{\ast}\delta v,v\rangle: v\in(\Ker\delta)^{\perp}, \
\|v\|=1\}>0.
$$

{The general Hodge decomposition \eqref{eq:gen_Hodge} implies that
$\Delta_\ell=\delta_\ell^*\delta_\ell$ acts on $\ker(\Delta_\ell)$  as the
zero operator (trivially), as $\delta^*_{\ell}\delta_\ell\colon
\overline{\im(\delta_\ell^*)}\to \im(\delta_{\ell}^*)$ (preserving this subspace) and as
  $\delta_{\ell-1}\delta_{\ell-1}^*$ on $\overline{\im(\delta_{\ell-1})}$,
  mapping also this subspace to itself.} 

{Now the image of an operator on a Hilbert space is closed if and only if it
maps the complement of its kernel isomorphically (with bounded inverse) to its
image. As the kernel of $\delta_{\ell}$ is the complement of the image of
$\delta_{\ell}^*$ and the kernel of $\delta_{\ell-1}^*$ is the complement of
the imaga of $\delta_\ell$, this implies indeed that $\Img(\Delta_\ell)$ is
closed if and only if (1) and (2) are satisfied.}

This finishes the proof of the lemma.
\end{proof}

\begin{corollary}
 For all $\ell\geq 0$ the following are isomorphisms,
  provided $\Img(\delta)$ is closed.
$$
\delta_{\ell}\colon  \Img\delta^{\ast}_{\ell}\to\Img\delta_{\ell}\ \ \ \
\text{and}\ \ \ \ \delta^{\ast}_{\ell}\colon
\Img\delta_{\ell}\to\Img\delta^{\ast}_{\ell}
$$
\end{corollary}
\begin{proof}
  The first map is injective because if $\delta (\delta^{\ast} f)=0$ then
  $0=\langle \delta\delta^{\ast} f,f\rangle=\|\delta^{\ast}f\|^2$ and so $\delta^\ast f=0$. It
  is surjective because of the decomposition (leaving out the subscripts)
$$
\delta (V)=\delta
(\Img\delta\oplus\Img\delta^{\ast}\oplus\Ker\Delta)=\delta(\Img\delta^{\ast})
$$
since $\delta$ is zero on the first and third summands of the left side of the
second equality. The argument for the second map is the same.
\end{proof}

The difficulty in applying the Hodge Lemma is in verifying that either
$\delta$ or $\delta^{\ast}$ has closed range. A sufficient condition is the
following, first pointed out to us by Shmuel Weinberger.
\begin{proposition}\label{PropositionWeinberger}
 Suppose that in the context of Lemma \ref{Lemma1}, the $L^2$-cohomology space $\Ker\delta_{\ell}/\Img\delta_{\ell-1}$ is
  finite dimensional. Then $\delta_{\ell-1}$ has closed range.
\end{proposition}
\begin{proof}
  We show more generally, that if $T\colon B\to V$ is a bounded linear map of Banach
  spaces, with $\Img T$ having finite codimension in $V$ then $\Img T$
  is closed in $V$. We can assume without loss of generality that $T$ is
  injective, by replacing $B$ with $B/\Ker T$ if necessary. Thus $T\colon
  B\to \Img T\oplus F=V$ where $\dim F<\infty$. Now define $G\colon
  B\oplus F\to V$ by $G(x,y)=Tx+y$. $G$ is bounded , surjective and injective,
  and thus an isomorphism by the open mapping theorem. Therefore $G(B)=T(B)$
  is closed in $V$.
\end{proof}

Consider the special case where $K(x,y)=1$ for all $x,y$ in
$X$. Let $\partial_{\ell}^0$ be the corresponding operator in
\eqref{2.4}. We have
\begin{lemma}\label{Lemma2}
  For $\ell>1$, $\Img\partial_{\ell}^0=\Ker\partial_{\ell-1}^0$,
  and $\Img\partial_1^0=\{1\}^{\perp}$ the orthogonal complement of the
  constants in $L^2(X)$.
\end{lemma}

Under that assumption $K\equiv1$, we can already finish the proof of
Theorem \ref{Theorem1}; the general case is proven later. Indeed
Lemma~\ref{Lemma2} implies that $\Img\partial_{\ell}$ is closed for
all $\ell$ since null spaces and orthogonal complements are closed,
and in fact shows that the homology \eqref{2.5} in this case is
trivial for $\ell>0$ and one dimensional for $\ell=0$.

\begin{proof}[Proof of Lemma \ref{Lemma2}]
 Let $h\in\{1\}^{\perp}\subset L^2(X)$. Define $g\in
  L^2(X^2)$ by $g(x,y)=h(y)$. Then from \eqref{2.4}
$$
\partial_1^0g(x_0)=\int_X(g(t,x_0)-g(x_0,t))\, d\mu(t)=\int_X (h(x_0)-h(t))\,
d\mu(t)=h(x_0)
$$
since $\mu(X)=1$ and $\int_X h\, d\mu=0$. It can be easily checked that
$\partial_1^0$ maps $L^2(X^2)$ into $\{1\}^{\perp}$, thus proving the lemma
for $\ell=1$. For $\ell>1$ let $h\in\Ker\partial_{\ell-1}^0$. Define
$g\in L^2(X^{\ell+1})$ by
$g(x_0,\dots,x_{\ell})=(-1)^{\ell}h(x_0,\dots,x_{\ell-1})$. Then, by \eqref{2.4}

\begin{align*}
\partial_{\ell}^0 g(x_0,\dots,x_{\ell-1}) &=\sum_{i=0}^{\ell}(-1)^i\int_X g(x_0,\dots,x_{i-1},t,x_i,\dots,x_{\ell-1})\, d\mu(t)\\
&=(-1)^{\ell}\sum_{i=0}^{\ell-1}(-1)^i\int_X h(x_0,\dots,x_{i-1},t,x_i,\dots,x_{\ell-2})\, d\mu(t)\\
&+(-1)^{2\ell}h(x_0,\dots,x_{\ell-1})\\
&=(-1)^{\ell}\partial_{\ell-1}^0 h(x_0,\dots,x_{\ell-2})+h(x_0,\dots,x_{\ell-1})\\
&=h(x_0,\dots,x_{\ell-1})
\end{align*}
since $\partial^0_{\ell-1} h=0$, finishing the proof.
\end{proof}

The next lemma give some general conditions on $K$ that guarantee
$\partial_{\ell}$ has closed range.
\begin{lemma}\label{Lemma3}
  Assume that $K(x,y)\geq \sigma>0$ for all $x,y\in X$. Then
  $\Img\partial_{\ell}$ is closed for all $\ell$. In fact,
  $\Img\partial_{\ell}=\Ker\partial_{\ell-1}$ for $\ell>1$ and has
  co-dimension one in $L^2(X)$ for $\ell=1$.
\end{lemma}

\begin{proof}
  Let $M_{\ell}\colon L^2(X^{\ell})\to L^2(X^{\ell})$ be the multiplication operator
$$M_{\ell}(f)(x_0,\dots,x_{\ell})=\prod_{j\neq k}\sqrt{K(x_j,x_k)} f(x_0,\dots,x_{\ell})
$$
Since $K\in L^{\infty}(X^2)$ and is bounded below by $\sigma$, $M_{\ell}$
clearly defines an isomorphism. The Lemma then follows from Lemma \ref{Lemma2}, and the
observation that
\[
\partial_{\ell}=M^{-1}_{\ell-1}\circ\partial^0_{\ell}\circ M_{\ell}.\qedhere
\]
\end{proof}

Theorem \ref{Theorem1} now follows from the Hodge Lemma and Lemma
\ref{Lemma3}. We note that Lemma \ref{Lemma2}, Lemma \ref{Lemma3} and
Theorem \ref{Theorem1} hold in the alternating setting, when
$L^2(X^{\ell})$ is replaced with $L^2_a(X^{\ell})$; so the cohomology
is also trivial in that setting.

For background, one could see Munkres \cite{17} for the algebraic topology,
Lang \cite{14} for the analysis, and Warner \cite{22} for the geometry.

\section{Metric spaces}\label{Section3}

For the rest of the paper, we assume that $X$ is a complete, separable metric
space, and that $\mu$ is a Borel probability measure on $X$, and $K$ is a
continuous function on $X^2$ (as well as symmetric, non-negative and bounded as
in Section \ref{Section2}). We will also assume throughout the rest of the
paper that
$\mu(U)>0$ for $U$ any nonempty open set.

The goal of this section is a Hodge Decomposition for continuous alternating
functions. Let $C(X^{\ell+1})$ denote the continuous functions on
$X^{\ell+1}$. We will use the following notation:
$$
C^{\ell+1}=C(X^{\ell+1})\cap L^2_a(X^{\ell+1})\cap L^{\infty}(X^{\ell+1}).
$$
Note that
$$
\delta\colon C^{\ell+1}\to C^{\ell+2}\ \ \text{and}\ \ \partial\colon  C^{\ell+1}\to
C^{\ell}
$$
are well defined linear maps. The only thing to check is that $\delta (f)$ and
$\partial (f)$ are continuous and bounded if $f\in C^{\ell+1}$. In the case of
$\delta (f)$ this is obvious from \eqref{2.3}. The following proposition from
analysis, \eqref{2.4} and the fact that $\mu$ is Borel imply that $\partial
(f)$ is
bounded and continuous.
\begin{proposition} 
Let $Y$ and $X$ be metric spaces, $\mu$ a
  Borel measure on $X$, and $M, g\in C(Y\times X)\cap L^{\infty}(Y\times
  X)$. Then $dg\in C(X)\cap L^{\infty}(X)$, where
$$
dg(x)=\int_X M(x,t)g(x,t)\, d\mu(t).
$$
\end{proposition}
\begin{proof}
 The fact that $dg$ is bounded follows easily from the
  definition and properties of $M$ and $g$, and continuity follows from a
  simple application of the Dominated Convergence Theorem, proving the
  proposition.
\end{proof}
Therefore we have the chain complexes:
\begin{equation*}
  \begin{CD}
    \cdots @>{\partial_{\ell+1}} >> C^{\ell+1} @> {\partial_{\ell}} >>
    C^{\ell} @> {\partial_{\ell -1}} >> \cdots C^1 @>{ \partial_0} >> 0
  \end{CD}
\end{equation*}
and
\begin{equation*}
\begin{CD}
  0 @>>> C^1 @>{\delta_0} >> C^2 @>{\delta_1} >> \cdots @>{\delta_{\ell-1} }>>
  C^{\ell+1} @>{\delta_{\ell}} >> \cdots
\end{CD}
\end{equation*}

\noindent In this setting we will prove

\begin{theorem}\label{Theorem2}
  Assume that $K$ satisfies the hypotheses of Theorem \ref{Theorem1}, and is
  continuous. Then we have the orthogonal (with respect to $L^2$), direct sum
  decomposition
$$
C^{\ell+1}=\delta (C^{\ell})\oplus \partial (C^{\ell+2})\oplus \Ker_C
\Delta
$$
where $\Ker_C \Delta$ denotes the subspace of elements in $\Ker
\Delta$ that are in $C^{\ell+1}$.
\end{theorem}

As in Theorem \ref{Theorem1}, the third summand is trivial except when
${\ell}=0$ in which
case it consists of the constant functions. We first assume that $K\equiv
1$. The proof follows from a few propositions. In the remainder of the
section, $\Img\delta$ and $\Img\partial$ will refer to the image
spaces of $\delta$ and $\partial$ as operators on $L^2_a$. The next
proposition gives formulas for $\partial$ and $\Delta$ on alternating
functions.

\begin{proposition}\label{Proposition5}
  For $f\in L^2_a(X^{\ell+1})$ we have
$$
\partial f(x_0,\dots,x_{\ell-1})=(\ell+1)\int_X f(t,x_0,\dots,x_{\ell-1})\,
d\mu(t)
$$
and
$$
\Delta
f(x_0,\dots,x_{\ell})=(\ell+2)f(x_0,\dots,x_{\ell})-\frac{1}{\ell+1}\sum_{i=0}^{\ell} \partial
f(x_0,\dots,\hat x_i,\dots,x_{\ell}).
$$
\end{proposition}

\begin{proof}
 The first formula follows immediately from \eqref{2.4} and the fact
  that $f$ is alternating. The second follows from a simple calculation using
  \eqref{2.3}, \eqref{2.4} and the fact that $f$ is alternating.
\end{proof}

Let $P_1$, $P_2$, and $P_3$ be the orthogonal projections implicit in Theorem
\ref{Theorem1}
$$
P_1\colon L^2_a(X^{\ell+1})\to \Img \delta, \ P_2\colon L^2_a(X^{\ell+1})\to
\Img \partial,\ \text{and} \ P_3\colon L^2_a(X^{\ell+1})\to\Ker\Delta
$$
\begin{proposition}\label{Proposition6} Let $f\in C^{\ell+1}$. Then $P_1(f)\in
  C^{\ell+1}$.
\end{proposition}
\begin{proof}
  It suffices to show that $P_1(f)$ is continuous and bounded. Let
  $g=P_1(f)$. It follows from Theorem \ref{Theorem1} that $\partial f=\partial g$, and
  therefore $\partial g$ is continuous and bounded. Since $\delta g=0$, we
  have, for $t,x_0,\dots,x_{\ell}\in X$
$$
0=\delta g(t,x_0,\dots,x_{\ell})=g(x_0,\dots,x_{\ell})-\sum_{i=0}^{\ell}
(-1)^i g(t,x_0,\dots,\hat x_i,\dots,x_{\ell}).
$$
Integrating over $t\in X$ gives us
\begin{align*} g(x_0,\dots,x_{\ell}) &= \int_X g(x_0,\dots,x_{\ell})\, d\mu(t)=\sum_{i=0}^{\ell} (-1)^i \int_X g(t,x_0,\dots,\hat x_i,\dots,x_{\ell})\, d\mu(t)\\
&=\frac{1}{\ell+1}\sum_{i=0}^{\ell} (-1)^i\partial g(x_0,\dots,\hat
x_i,\dots,x_{\ell}).
\end{align*}
As $\partial g$ is continuous and bounded, this implies $g$ is continuous and
bounded.
\end{proof}

\begin{corollary}
  If $f\in C^{\ell+1}$, then $P_2(f)\in C^{\ell+1}$.
\end{corollary}

This follows from the Hodge decomposition (Theorem \ref{Theorem1}) and the
fact that
$P_3(f)$ is continuous and bounded (being a constant).

The following proposition can be thought of as analogous to a regularity
result in elliptic PDE's. It states that solutions to $\Delta u=f$, $f$
continuous, which are \emph{a priori} in $L^2$ are actually continuous.

\begin{proposition}\label{Proposition7}
  If $f\in C^{\ell+1}$ and $\Delta u=f$, $u\in L^2_a(X^{\ell+1})$ then $u\in
  C^{\ell+1}$.
\end{proposition}
\begin{proof}
  From Proposition \ref{Proposition5}, (with $u$ in place of $f$) we have
\begin{align*}
\Delta u(x_0,\dots,x_{\ell}) &=(\ell+2)u(x_0,\dots,x_{\ell})-\frac{1}{\ell+1}\sum_{i=0}^{\ell} \partial u(x_0,\dots,\hat x_i,\dots,x_{\ell})\\
&=f(x_0,\dots,x_{\ell})
\end{align*}
and solving for $u$, we get
$$
u(x_0,\dots,x_{\ell})=\frac{1}{\ell+2}
f(x_0,\dots,x_{\ell})+\frac{1}{(\ell+2)(\ell+1)}\sum_{i=0}^{\ell} \partial
u(x_0,\dots,\hat x_i,\dots,x_{\ell}).
$$
It therefore suffices to show that $\partial u$ is continuous and
bounded. However, it is easy to check that $\Delta \circ \partial=\partial
\circ \Delta$ and thus
$$
\Delta (\partial u)=\partial \Delta u=\partial f
$$
is continuous and bounded. But then, again using Proposition
\ref{Proposition5},
\begin{align*}
\Delta(\partial u)(x_0,\dots,x_{\ell-1}) &=(\ell+1)\partial u (x_0,\dots,x_{\ell-1})\\
&-\frac{1}{\ell}\sum_{i=0}^{\ell-1}(-1)^i \partial (\partial u)(x_0,\dots,\hat
x_i,\dots,x_{\ell-1})
\end{align*}
and so, using $\partial^2=0$ we get
$$
(\ell+1)\partial u=\partial f
$$
which implies that $\partial u$ is continuous and bounded, finishing the
proof.
\end{proof}

  \begin{proposition}\label{Proposition8} If $g\in C^{\ell+1}\cap \Img
    \delta$, then $g=\delta h$ for some $h\in C^{\ell}$.
\end{proposition}

\begin{proof}
  From the corollary of the Hodge Lemma, let $h$ be the unique element in
  $\Img \partial$ with $g=\delta h$. Now $\partial g$ is continuous and
  bounded, and
$$
\partial g=\partial\delta h=\partial\delta h+\delta\partial h=\Delta h
$$
since $\partial h=0$. But now $h$ is continuous and bounded from Proposition
\ref{Proposition7}.
\end{proof}

  \begin{proposition}
    \label{Proposition9} If $g\in C^{\ell+1}\cap L^2_a(X^{\ell+1})$, the
    $g=\partial h$ for some $h\in C^{\ell+2}$.
\end{proposition}

The proof is identical to the one for Proposition \ref{Proposition8}.

Theorem \ref{Theorem2}, in the case $K\equiv 1$ now follows from Propositions
\ref{Proposition6} through
\ref{Proposition9}. The proof easily extends to general $K$ which is bounded
below by a positive constant.

\section[Hodge Theory at Scale alpha]{\boldmath Hodge Theory at Scale $\alpha$}\label{Section4}

As seen in Sections \ref{Section2} and \ref{Section3}, the chain and cochain
complexes constructed on
the whole space yield trivial cohomology groups. In order to have a theory
that gives us topological information about $X$, we define our complexes on a
neighborhood of the diagonal, and restrict the boundary and coboundary
operator to these complexes. The corresponding cohomology can be considered a
cohomology of $X$ at a scale, with the scale being the size of the
neighborhood. We will assume throughout this section that $(X,d)$ is a compact
metric space. For $x,y\in X^{\ell}$, $\ell>1$, this induces a metric
compatible with the product topology
$$
d_{\ell}(x,y)=\max\{d(x_0,y_0),\dots d(x_{\ell-1},y_{\ell-1})\}.
$$

The diagonal $D_{\ell}$ of $X^{\ell}$ is just $\{x\in
X^{\ell}:x_i=x_j,\ i,j=0,\dots,\ell-1\}$ For $\alpha>0$ we let
$U_\alpha^\ell$ be the $\alpha$-neighborhood of the diagonal in $X^\ell$, namely
\begin{align*}
U^{\ell}_{\alpha}& =\{x\in X^{\ell}:d_{\ell}(x,D_{\ell})\leq\alpha\}\\
&=\{x\in X^{\ell}: \exists t\in X\ \text{such that}\ d(x_i,t)\leq\alpha,\
i=0,\dots,\ell-1\}.
\end{align*}

Observe that $U^{\ell}_{\alpha}$ is closed and that for $\alpha\geq$ diameter
$X$, $U^{\ell}_{\alpha}=X^{\ell}$.

One could, alternatively, have defined neighbourhoods $V_\alpha^\ell$
as those $x\in X^\ell$ such that $d(x_i,x_j)\le\alpha$ whenever $0\le
i,j<\ell$; this definition appears in the Vietoris-Rips complex, see
Remark~\ref{remVietoris}. Both definitions are very close, in the
sense that $V_\alpha^\ell\subseteq U_\alpha^\ell\subseteq
V_{2\alpha}^\ell$.

The measure $\mu_{\ell}$ induces a Borel measure on $U^{\ell}_{\alpha}$ which
we will simply denote by $\mu_{\ell}$ (not a probability measure). For
simplicity, we will take $K\equiv 1$ throughout this section, and consider
only alternating functions in our complexes. We first discuss the $L^2$-theory, and thus our basic spaces will be $L^2_{a}(U^{\ell}_{\alpha})$, the
space of alternating functions on $U^{\ell}_{\alpha}$ that are in $L^2$ with
respect to $\mu_{\ell}$, $\ell>0$. Note that if $(x_0,\dots,x_{\ell})\in
U^{\ell+1}_{\alpha}$, then $(x_0,\dots,\hat x_i,\dots,x_{\ell})\in
U^{\ell}_{\alpha}$ for $i=0,\dots,\ell$. It follows that if $f\in
L^2_a(U^{\ell}_{\alpha})$, then $\delta f\in L^2_a(U^{\ell+1}_{\alpha})$. We
therefore have the well defined cochain complex
$$
\begin{CD}
  0 @>>> L^2_a(U^1_{\alpha}) @>{\delta} >> L^2_a(U^2_{\alpha})
  \cdots @>{\delta }>> L^2_a(U^{\ell}_{\alpha}) @>{\delta} >>
  L^2_a(U^{\ell+1}_{\alpha})\cdots
\end{CD}
$$
Since $\partial=\delta^{\ast}$ depends on the integral, the expression for it
will be different from \eqref{2.4}. We define a ``slice'' by
$$
S_{x_0\cdots x_{\ell-1}}=\{t\in X\colon (x_0,\dots,x_{\ell-1},t)\in
U^{\ell+1}_{\alpha}\}.
$$
We note that, for $S_{x_0\cdots x_{\ell-1}}$ to be nonempty,
$(x_0,\dots,x_{\ell-1})$ must be in $U^{\ell}_{\alpha}$. Furthermore
$$
U^{\ell+1}_{\alpha}=\{(x_0,\dots,x_{\ell}):(x_0,\dots,x_{\ell-1})\in
U^{\ell}_{\alpha},\ \text{and} \ x_{\ell}\in S_{x_0\cdots x_{\ell-1}}\}.
$$
It follows from the proof of Proposition \ref{Proposition1} of Section
\ref{Section2} and the fact that
$K\equiv 1$, that $\delta\colon L^2_a(U^{\ell}_{\alpha})\to
L^2_a(U^{\ell+1}_{\alpha})$ is bounded and that $\|\delta\|\leq \ell+1$, and
therefore $\delta^{\ast}$ is bounded.  The adjoint of the operator
$\delta\colon L^2_a(U^{\ell}_{\alpha})\to L^2_a(U^{\ell+1}_{\alpha})$ will be
denoted, as before, by either $\partial$ or $\delta^{\ast}$ (without the
subscript $\ell$).

\begin{proposition}
  \label{Proposition10} For $f\in L^2_a(U^{\ell+1}_{\alpha})$ we have
$$
\partial f(x_0,\dots,x_{\ell-1})=(\ell+1)\int_{S_{x_0\cdots x_{\ell-1}}}
f(t,x_0,\dots,x_{\ell-1})\, d\mu(t).
$$
\end{proposition}

\begin{proof}
  The proof is essentially the same as the proof of Proposition
  \ref{Proposition4}, using the
  fact that $K\equiv 1$, $f$ is alternating, and the above remark.
\end{proof}

It is worth noting that the domain of integration depends on $x\in
U^{\ell}_{\alpha}$, and this makes the subsequent analysis more difficult than
in Section \ref{Section3}. We thus have the corresponding chain complex
 $$
 \begin{CD}
   \cdots @>{\partial} >> L^2_a(U^{\ell+1}_{\alpha}) @> {\partial} >>
   L^2_a(U^{\ell}_{\alpha}) @> {\partial} >> \cdots L^2_a(U^1_{\alpha})
   @>{ \partial} >> 0.
 \end{CD}
 $$
 Of course, $U^1_{\alpha}=X$. The corresponding Hodge Laplacian is the
 operator $\Delta\colon L^2_a(U^{\ell}_{\alpha})\to L^2_a(U^{\ell}_{\alpha})$,
 $\Delta=\partial\delta+\delta\partial$, where all of these operators depend
 on $\ell$ and $\alpha$. When we want to emphasize this dependence, we will
 list $\ell$ and (or) $\alpha$ as subscripts. We will use the following
 notation for the cohomology and harmonic functions of the above complexes:
 $$
 H^{\ell}_{L^2,\alpha}(X)=\frac{\Ker\delta_{\ell,\alpha}}{\Img\
   \delta_{\ell-1,\alpha}}\ \ \ \text{and}\ \ \
 \Harm^{\ell}_{\alpha}(X)=\Ker\Delta_{\ell,\alpha}.
 $$
 \begin{remark}
   If $\alpha\geq\diam(X)$, then $U^{\ell}_{\alpha}=X^{\ell}$, so the
   situation is as in Theorem \ref{Theorem1} of Section \ref{Section2}, so $H^{\ell}_{L^2,\alpha}(X)=0$
   for $\ell>0$ and $H^0_{L^2,\alpha}(X)=\mathbb{R}$. Also, if $X$ is a finite
   union of connected components $X_1,\dots,X_k$, and
   $\alpha<d(X_i,X_j)$ for all $i\neq j$, then
   $H^{\ell}_{L^2,\alpha}(X)=\bigoplus_{i=1}^k H^{\ell}_{L^2,\alpha}(X_i)$.
 \end{remark}

 \begin{definition}
   We say that Hodge theory for $X$ at scale $\alpha$ holds if we have the
   orthogonal direct sum decomposition into closed subspaces
 $$
 L^2_a(U^{\ell}_{\alpha})=\Img\delta_{\ell-1}\oplus \Img\
 \delta^{\ast}_{\ell}\oplus\Harm^{\ell}_{\alpha}(X) \ \ \text{for all}\ \ell
 $$
 and furthermore, $H^{\ell}_{\alpha,L^2}(X)$ is isomorphic to
 $\Harm^{\ell}_{\alpha}(X)$, with each equivalence class in the former
 having a unique representative in the latter.
\end{definition}

\begin{remark}\label{remFunctorial}
  Hodge theory is functorial, in the sense that, for any $s\ge1$, the
  inclusion $U_\alpha^\ell\subseteq U_{s\alpha}^\ell$ induces
  corestriction maps $H_{s\alpha}^\ell\to H_\alpha^\ell$. In seeking a
  robust notion of cohomology, it will make sense to consider the
  images of these maps at a sufficiently large separation $s$, rather
  than at individual cohomology groups $H_\alpha^\ell$.

  !!!! richer kind of functoriality, for maps $f:Y\to X$? Which conditions on $f$?
\end{remark}

\begin{theorem}\label{Theorem3} If $X$ is a compact metric space, $\alpha>0$,
  and the $L^2$-cohomology spaces
$\Ker\delta_{\ell,\alpha}/\Img\delta_{\ell-1,\alpha}$, $\ell\geq 0$
are finite dimensional, then Hodge theory for $X$ at scale $\alpha$ holds.
\end{theorem}
\begin{proof}
 This is immediate from the Hodge Lemma (Lemma \ref{Lemma1}), using Proposition~\ref{PropositionWeinberger} from Section \ref{Section2}.
\end{proof}

 We record the formulas for $\delta\partial f$ and $\partial\delta f$ for
 $f\in L^2_a(U^{\ell+1}_{\alpha})$

\begin{multline*} \delta(\partial
f)(x_0,\dots,x_{\ell})\\=(\ell+1)\sum_{i=0}^{\ell}(-1)^i\int_{S_{x_0,\dots,\hat
    x_i,\dots,x_{\ell}}}f(t,x_0,\dots,\hat
x_i,\dots,x_{\ell})d\mu(t) \end{multline*}

\begin{multline}\label{4.4}
\partial(\delta f)(x_0,\dots,x_{\ell})=(\ell+2)\mu(S_{x_0,\dots,x_{\ell}})f(x_0,\dots,x_{\ell})\\
+(\ell+2)\sum_{i=0}^{\ell}(-1)^{i+1}\int_{S_{x_0,\dots,x_{\ell}}}f(t,x_0,\dots,\hat
x_i,\dots,x_{\ell})d\mu(t) \end{multline}

Of course, the formula for $\Delta f$ is found by adding these two.

\begin{remark}
  Harmonic forms are solutions of the optimization problem: Minimize the
  ``Dirichlet norm'' $\|\delta f\|^2+\|\partial f\|^2=\langle \Delta
  f,f\rangle=\langle \Delta^{1/2}f,\Delta^{1/2}f\rangle$ over $f\in L^2_a(U^{\ell+1}_{\alpha})$.
\end{remark}

\begin{remark}\label{remVietoris}
  The alternative neighbourhoods $V^{\ell+1}_{\alpha}$, giving rise to
  the Vietoris-Rips complex (see Chazal and Oudot \cite{4}), were
  defined by $(x_0,\dots,x_{\ell})\in U^{\ell+1}_{\alpha}$ if and only
  if $d(x_i,x_j)\leq\alpha$ for all $i,j$. This corresponds to the
  theory developed in Section \ref{Section2} with $K(x,y)$ equal to
  the characteristic function of $V^2_{\alpha}$. A version of Theorem
  \ref{Theorem3} holds in this case.
\end{remark}

\section[L2-Theory of alpha-Harmonic 0-Forms]{\boldmath $L^2$-Theory of $\alpha$-Harmonic $0$-Forms}\label{Section5}

In this section we assume that we are in the setting of Section
\ref{Section4}, with
$\ell=0$. Thus $X$ is a compact metric space with a probability measure and
with a fixed scale $\alpha>0$.

Recall that $f\in L^2(X)$ is $\alpha$-harmonic if
$\Delta_{\alpha}f=0$. Moreover, if $\delta\colon L^2(X)\to L_a^2(U^2_{\alpha})$
denotes the coboundary, then $\Delta_{\alpha}f=0$ if and only if $\delta f=0$;
also $\delta f(x_0,x_1)=f(x_1)-f(x_0)$ for all pairs $(x_0,x_1)\in
U^2_{\alpha}$.

Recall that for any $x\in X$, the slice $S_{x,\alpha}=S_x\subset X^2$ is the
set
$$
S_x=S_{x,\alpha}=\{t\in X: \exists p\in X\ \text{such that} \ x,t\in
B_{\alpha}(p)\} .
$$
Note that $B_{\alpha}(x)\subset S_{x,\alpha}\subset B_{2\alpha}(x)$. It
follows that $x_1\in S_{x_0,\alpha}$ if and only if $x_0\in
S_{x_1,\alpha}$. We conclude
\begin{proposition}
  Let $f\in L^2(X)$. Then $\Delta_{\alpha}f=0$ if and only if $f$ is locally
  constant in the sense that $f$ is constant on $S_{x,\alpha}$ for every $x\in
  X$. Moreover if $\Delta_{\alpha}f=0$, then
  \begin{enumerate}
    \item[(a)] If $X$ is connected, then $f$ is constant.
    \item[(b)] If $\alpha$ is greater than the maximum distance between components of
    $X$, then $f$ is constant.
    \item[(c)] For any $x\in X$, $f(x)=$average of $f$ on $S_{x,\alpha}$ and on
    $B_{\alpha}(x)$.
    \item[(d)] Harmonic functions are continuous.
  \end{enumerate}
\end{proposition}

We note that continuity of $f$ follows from the fact that $f$ is constant on
each slice $S_{x,\alpha}$, and thus locally constant.

\begin{remark}
  We will show that (d) is also true for harmonic 1-forms with an additional
  assumption on $\mu$, (Section \ref{Section8}) but are unable to prove it for
  harmonic 
  2-forms.
\end{remark}
Consider next an extension of (d) to the Poisson regularity problem. If
$\Delta_{\alpha}f=g$ is continuous, is $f$ continuous? In general the answer
is no, and we will give an example.

Since $\partial_0$ on $L^2(X)$ is zero, the $L^2$-$\alpha$-Hodge theory
(Section \ref{Section9}) takes the form
$$
L^2(X)=\Img\partial \oplus\Harm_{\alpha},
$$
where $\partial\colon L^2(U^2_{\alpha})\to L^2(X)$ and $\Delta f=\partial\delta
f$. Thus for $f\in L^2(X)$, by~\eqref{4.4}
\begin{equation}\label{star}
\Delta_{\alpha}f(x)=2\mu(S_{x,\alpha})f(x)-2\int_{S_{x,\alpha}}f(t)\, d\mu(t)
\end{equation}

The following example shows that an additional assumption is needed for the
Poisson regularity problem to have an affirmative solution. Let $X$ be the
closed interval $[-1,1]$ with the usual metric $d$ and let $\mu$ be the
Lebesgue measure on $X$ with an atom at $0$, $\mu(\{0\})=1$. Fix any
$\alpha<1/4$. We will define a piecewise linear function on $X$ with
discontinuities at $-2\alpha$ and $2\alpha$ as follows. Let $a$ and $b$ be any
real numbers $a\neq b$, and define
\begin{equation*}
f(x)=
\begin{cases}
  \frac{a-b}{8\alpha}+a, & -1\leq x<-2\alpha\\
     \frac{b-a}{4\alpha}(x-2\alpha)+b, &  -2\alpha\leq x\leq 2\alpha\\
     \frac{a-b}{8\alpha}+b,& 2\alpha<x\leq 1 .
   \end{cases}
\end{equation*}

Using \eqref{star} above one readily checks that $\Delta_{\alpha}f$ is
continuous by
computing left hand and right hand limits at $\pm 2\alpha$. (The constant
values of $f$ outside $[-2\alpha,2\alpha]$ are chosen precisely so that the
discontinuities of the two terms on the right side of~\eqref{star} cancel out.)

With an additional ``regularity'' hypothesis imposed on $\mu$, the Poisson
regularity property holds. In the rest of this section assume that
$\mu(S_x\cap A)$ is a continuous function of $x\in X$ for each measurable set
$A$. One can show that if $\mu$ is Borel regular, then this will hold provided
$\mu(S_x\cap A)$ is continuous for all closed sets $A$ (or all open sets $A$).
\begin{proposition}\label{PoissonRegularity}
  Assume that
$\mu(S_x\cap A)$ is a continuous function of $x\in X$ for each measurable set
$A$. If $\Delta_{\alpha}f=g$ is continuous for $f\in L^2(X)$ then $f$ is continuous.
\end{proposition}
\begin{proof}
  From~\eqref{star} we have
$$
f(x)=\frac{g(x)}{2\mu(S_x)}+\frac{1}{\mu(S_x)}\int_{S_x}f(t)\,d\mu(t)
$$
The first term on the right is clearly continuous by our hypotheses on $\mu$
and the fact that $g$ is continuous. It suffices to show that the function
$h(x)=\int_{S_x}f(t)\, d\mu(t)$ is continuous. If $f=\chi_{ _A}$ is the
characteristic function of any measurable set $A$, then $h(x)=\mu(S_x\cap A)$
is continuous, and therefore $h$ is continuous for $f$ any simple function
(linear combination of characteristic functions of measurable sets). From
general measure theory, if $f\in L^2(X)$, we can find a sequence of simple
functions $f_n$ such that $f_n(t)\to f(t)$ a.e, and $|f_n(t)|\leq |f(t)|$ for
all $t\in X$. Thus $h_n(x)=\int_{S_x}f_n(t)\, d\mu(t)$ is continuous and
$$
|h_n(x)-h(x)|\leq \int_{S_x}|f_n(t)-f(t)|\, d\mu(t)\leq
\int_{X}|f_n(t)-f(t)|\, d\mu(t)
$$
Since $|f_n- f|\to 0$ a.e, and $|f_n-f|\leq 2|f|$ with $f$ being in $L^1(X)$,
it follows from the dominated convergence theorem that $\int_X|f_n-f|\,d\mu\to
0$. Thus $h_n$ converges uniformly to $h$ and so continuity of $h$ follows
from continuity of $h_n$.
\end{proof}

We don't have a similar result for 1-forms.

Partly to relate our framework of $\alpha$-harmonic theory to some previous
work, we combine the setting of Section \ref{Section2} with Section
\ref{Section4}. Thus we now put back 
the function $K$. Assume $K> 0$ is a symmetric and continuous function
$K\colon X\times X\to \mathbb{R}$, and $\delta$ and $\partial$ are defined as in
Section \ref{Section2}, but use a similar extension to general $\alpha>0$, of
Section \ref{Section4},
all in the $L^2$-theory.

Let $D\colon L^2(X)\to L^2(X)$ be the operator defined as multiplication by
the function
$$
D(x)=\int_X G(x,y)\, d\mu(y)\ \ \ \ \ \text{where}\
G(x,y)=K(x,y)\chi_{U^2_{\alpha}}
$$
using the characteristic function $\chi_{U^2_{\alpha}}$ of $U^2_{\alpha}$. So
$\chi_{U^2_{\alpha}}(x_0,x_1)=1$ if $(x_0,x_1)\in U^2_{\alpha}$ and 0
otherwise. Furthermore, let $L_G\colon L^2(X)\to L^2(X)$ be the integral operator
defined by
$$
L_G f(x)=\int_X G(x,y)f(y)\,d\mu(y).
$$
Note that $L_G(1)=D$ where $1$ is the constant function. When $X$ is compact
$L_G$ is a Hilbert-Schmidt operator (this was first noted to us by Ding-Xuan
Zhou). Thus $L_G$ is trace class and self adjoint. It is not difficult to see
now that ~\eqref{star} takes the form

\begin{equation}\label{starstar}
\frac{1}{2}\Delta_{\alpha} f=Df-L_G f.
\end{equation}

(For the special case $\alpha=\infty$, i.e. $\alpha$ is irrelevant as in
Section \ref{Section2}, this is the situation as in Smale and Zhou \cite{19} for
the case $K$
is a reproducing kernel.) As in the previous proposition:
\begin{proposition}
  The Poisson Regularity Property holds for the operator
  of \eqref{starstar}.
\end{proposition}

To get a better understanding of \eqref{starstar} it is useful to define a
normalization
of the kernel $G$ and the operator $L_G$ as follows. Let $\hat G\colon X\times
X\to\mathbb{R}$ be defined by
$$
\hat G(x,y)=\frac{G(x,y)}{(D(x)D(y))^{1/2}}
$$
and $L_{\hat G}\colon L^2(X)\to L^2(X)$ be the corresponding integral operator. Then
$L_{\hat G}$ is trace class, self adjoint, with non-negative eigenvalues, and
has a complete orthonormal system of continuous eigenfunctions.

!!!! referee thinks $L_{\hat G}$ is a reproducing kernel, but sees this as contradicting the next paragraph

A normalized $\alpha$-Laplacian may be defined on $L^2(X)$ by
$$
\frac{1}{2}\hat{\Delta}=I-L_{\hat G}
$$
so that the spectral theory of $L_{\hat G}$ may be transferred to
$\hat{\Delta}$. (Also, one might consider $\frac{1}{2}\Delta^{*}=I-D^{-1}L_G$
as in Belkin, De Vito, and Rosasco \cite{1}.)

In Smale and Zhou \cite{19}, for $\alpha=\infty$, error estimates are given
(reproducing kernel case) for the spectral theory of $L_{\hat G}$ in terms of
finite dimensional approximations. See especially Belkin and Niyogi \cite{2}
for limit theorems as $\alpha\to 0$.

\section{Harmonic forms on constant curvature manifolds}\label{Section6}

In this section we will give an explicit description of harmonic forms in a
special case. Let $X$ be a compact, connected, oriented manifold of dimension
$n>0$, with a Riemannian metric $g$ of constant sectional curvature. Also,
assume that $g$ is normalized so that $\mu(X)=1$ where $\mu$ is the measure
induced by the volume form associated with $g$, and let $d$ be the metric on
$X$ induced by $g$. Let $\alpha>0$ be sufficiently small so that for all $p\in
X$, the ball $B_{2\alpha}(p)$ is geodesically convex. That is, for $x,y\in
B_{2\alpha}(p)$ there is a unique, length minimizing geodesic $\gamma$ from
$x$ to $y$, and $\gamma$ lies in $B_{2\alpha}(p)$. Note that if
$(x_0,\dots,x_n)\in U^{n+1}_{\alpha}$, then $d(x_i,x_j)\leq 2\alpha$ for all
$i,j$, and thus all $x_i$ lie in a common geodesically convex ball. Such a
point defines an $n$-simplex with vertices $x_0,\dots,x_n$ whose faces are
totally geodesic submanifolds, which we will denote by
$\sigma(x_0,\dots,x_n)$. We will also denote the $k$-dimensional faces by
$\sigma(x_{i_0},\dots,x_{i_k})$ for $k<n$. Thus $\sigma(x_i,x_j)$ is the
geodesic segment from $x_i$ to $x_j$, $\sigma(x_i,x_j,x_k)$ is the union of
geodesic segments from $x_i$ to points on $\sigma(x_j,x_k)$ and higher
dimensional simplices are defined inductively. (Since $X$ has constant
curvature, this construction is symmetric in $x_0,\dots,x_n$.) A $k$-dimensional face will be called degenerate if one of its vertices is
contained in one of its $(k-1)$-dimensional faces.

Note that cohomology of the Vietoris-Rips complex has already been
considered by Hausmann~\cite{hausmann}, but his construction is quite
different from ours. He considers the limit, as $\epsilon\to0$, of the
simplicial cohomology of $X_\epsilon$. First, we contend that
important information is visible in $X_\alpha$ at particular scales
$\alpha$, possibly determined by the problem at hand, and not tending
to $0$. Second, Hausmann considers simplicial homology, with arbitrary
coefficients, while we consider $\ell^2$ cohomology, with real or
complex coefficients.

For $(x_0,\dots,x_n)\in U^{n+1}_{\alpha}$, the orientation on $X$
induces an orientation on $\sigma(x_0,\dots,x_n)$ (assuming it is
non-degenerate). For example, if $v_1,\dots,v_n$ denote the tangent
vectors at $x_0$ to the geodesics from $x_0$ to $x_1,\dots,x_n$, we
can define $\sigma(x_0,\dots,x_n)$ to be positive (negative) if
$\{v_1,\dots,v_n\}$ is a positive (respectively negative) basis for
the tangent space at $x_0$. Of course, if $\tau$ is a permutation, the
orientation of $\sigma(x_0,\dots,n)$ is equal to $(-1)^{\sign\tau}$
times the orientation of $\sigma(x_{\tau(0)},\dots,x_{\tau(n)})$. We
now define $f\colon U^{\ell+1}_{\alpha}\to\mathbb{R}$ by
\begin{align*}
f(x_0,\dots,x_n) &=\mu(\sigma(x_0,\dots,x_n))\ \ \text{for}\ \sigma(x_0,\dots,x_n)\ \text{positive}\\
&=-\mu(\sigma(x_0,\dots,x_n))\ \ \text{for}\ \sigma(x_0,\dots,x_n)\ \text{negative}\\
&=0\ \ \text{for}\ \sigma(x_0,\dots,x_n)\ \text{degenerate}.
\end{align*}

Thus $f$ is the signed volume of oriented geodesic $n$-simplices. Clearly $f$
is continuous as non-degeneracy is an open condition and the volume of a
simplex varies continuously in the vertices.

Recall that, in classical Hodge theory, every de Rham cohomology class
has a unique harmonic representative. In particular, the volume form
is harmonic, and generates top-dimensional cohomology. In our more
elaborate context, we can also pinpoint the ``form'' generating
top-dimensional cohomology. (See Remark~\ref{remCurvature} below on
relaxing the constant curvature hypothesis.) The main result of this
section is:

\begin{theorem}
  Let $X$ be a oriented Riemannian $n$-manifold of constant sectional
  curvature and $f$, $\alpha$ as above. Then $f$ is harmonic. In fact $f$ is the unique
  harmonic $n$-form in   $L^2_a(U^{n+1}_{\alpha})$ up to scaling.
\end{theorem}

\begin{proof}
  Uniqueness follows from Section \ref{Section9}. We will show that $\partial
  f=0$ and
  $\delta f=0$. Let $(x_0,\dots,x_{n-1})\in U^{n}_{\alpha}$. To show $\partial
  f=0$, it suffices to show, by Proposition \ref{Proposition10}, that
  \begin{equation}\label{6.1}
\int_{S_{x_0\cdots x_{n-1}}} f(t,x_0,\dots,x_{n-1})\, d\mu(t)=0.
\end{equation}

We may assume that $\sigma(x_0,\dots,x_{n-1})$ is non-degenerate, otherwise
the integrand is identically zero. Recall that $S_{x_0\cdots x_{n-1}}=\{t\in
X: (t,x_0,\dots,x_{n-1})\in U^{n+1}_{\alpha}\}\subset B_{2\alpha}(x_0)$
where $B_{2\alpha}(x_0)$ is the geodesic ball of radius $2\alpha$ centered at
$x_0$. Let $\Gamma$ be the intersection of the totally geodesic $n-1$
dimensional submanifold containing $x_0,\dots,x_{n-1}$ with
$B_{2\alpha}(x_0)$. Thus $\Gamma$ divides $B_{2\alpha}(x_0)$ into two pieces
$B^+$ and $B^-$. For $t\in \Gamma$, the simplex $\sigma(t,x_0,\dots,x_{n-1})$
is degenerate and therefore the orientation is constant on each of $B^+$ and
$B^-$, and we can assume that the orientation of $\sigma(t,x_0,\dots,x_{n-1})$
is positive on $B^+$ and negative on $B^-$. For $x\in B_{2\alpha}(x_0)$ define
$\phi(x)$ to be the reflection of $x$ across $\Gamma$. Thus the geodesic
segment from $x$ to $\phi(x)$ intersects $\Gamma$ perpendicularly at its
midpoint. Because $X$ has constant curvature, $\phi$ is a local isometry and
since $x_0\in\Gamma$, $d(x,x_0)=d(\phi(x),x_0)$. Therefore
$\phi\colon B_{2\alpha}(x_0)\to B_{2\alpha}(x_0)$ is an isometry which maps $B^+$
isometrically onto $B^-$ and $B^-$ onto $B^+$. Denote $S_{x_0\cdots x_{n-1}}$
by $S$. It is easy to see that $\phi\colon S\to S$, and so defining $S^{\pm}=S\cap
B^{\pm}$ it follows that $\phi\colon S^+\to S^-$ and $\phi\colon S^-\to S^+$ are
isometries. Now

\begin{align*}
\int_{S_{x_0\cdots x_{n-1}}} &f(t,x_0,\dots,x_{n-1})\, d\mu(t)\\
&=\int_{S^+} f(t,x_0,\dots,x_{n-1})\, d\mu(t)+\int_{S^-}f(t,x_0,\dots,x_{n-1})\, d\mu(t)\\
&=\int_{S^+}\mu(\sigma(t,x_0,\dots,x_{n-1}))\,
d\mu(t)-\int_{S^-}\mu(\sigma(t,x_0,\dots,x_{n-1}))\, d\mu(t).
\end{align*}

Since
$\mu(\sigma(t,x_0,\dots,x_{n-1}))=\mu(\sigma(\phi(t)t,x_0,\dots,x_{n-1}))$ for
$t\in S^+$, the last two terms on the right side cancel establishing
\eqref{6.1}. 

We now show that $\delta f=0$. Let $(t,x_0,\dots,x_n)\in
U^{n+2}_{\alpha}$. Thus
$$
\delta f(t,x_0,\dots,x_n)=f(x_0,\dots,x_n)+\sum_{i=0}^n
(-1)^{i+1}f(t,x_0,\dots,\hat x_i,\dots,x_n)
$$
and we must show that
\begin{equation}\label{6.2}
 f(x_0,\dots,x_n)=\sum_{i=0}^n (-1)^i f(t,x_0,\dots,\hat x_i,\dots,x_n).
\end{equation}
Without loss of generality, we will assume that
$\sigma(x_0,\dots,x_n)$ is positive. The demonstration of \eqref{6.2}
depends on the location of $t$. Suppose that $t$ is in the interior of
the simplex $\sigma(x_0,\dots,x_n)$. Then for each $i$, the
orientation of $\sigma(x_0,\dots,x_{i-1},t,x_{i+1},\dots,x_n)$ is the
same as the orientation of $\sigma(x_0,\dots,x_n)$ since $t$ and $x_i$
lie on the same side of the face $\sigma(x_0.\dots,\hat
x_i,\dots,x_n)$, and is thus positive. On the other hand, the
orientation of $\sigma(t,x_0,\dots\hat x_i,\dots,x_n)$ is $(-1)^i$
times the orientation of
$\sigma(x_0,\dots,x_{i-1},t,x_{i+1},\dots,x_n)$. Therefore the right
side of \eqref{6.2} becomes
$$
\sum_{i=0}^n \mu(\sigma(x_0,\dots,x_{i-1},t,x_{i+1},\dots,x_n)).
$$
This however equals $\mu(\sigma(x_0,\dots,x_n))$ which is the left side of
\eqref{6.2}, since
$$
\sigma(x_0,\dots,x_n)=\bigcup_{i=0}^n\sigma(x_0,\dots,x_{i-1},t,x_{i+1},\dots,x_n)
$$
when $t$ is interior to $\sigma(x_0,\dots,x_n)$.

There are several cases when $t$ is exterior to $\sigma(x_0,\dots,x_n)$ (or on
one of the faces), depending on which side of the various faces it lies. We
just give the details of one of these, the others being similar. Simplifying
notation, let $F_i$ denote the face ``opposite'' $x_i$, $\sigma(x_0,\dots,\hat
x_i,\dots,x_n)$, and suppose that $t$ is on the opposite side of $F_0$ from
$x_0$, but on the same side of $F_i$ as $x_i$ for $i\neq 0$. As in the above
argument, the orientation of $\sigma(x_0,\dots,x_{i-1},t,x_{i+1},\dots,x_n)$
is positive for $i\neq 0$ and is negative for $i=0$. Therefore the right side
of \eqref{6.2} is equal to

\begin{equation}\label{6.3}
\sum_{i=1}^n\mu(\sigma(x_0,\dots,x_{i-1},t,x_{i+1},\dots,x_n))-\mu(\sigma(t,x_1,\dots,x_n)).
\end{equation}

Let $s$ be the point where the geodesic from $x_0$ to $t$ intersects
$F_0$. Then for each $i>0$

\begin{multline*}
\sigma(x_0,\dots,x_{i-1},t,x_{i+1},\dots,x_n)=\sigma(x_0,\dots,x_{i-1},s,x_{i+1},\dots,x_n)\\
\cup\sigma(s,\dots,x_{i-1},t,x_{i+1},\dots,x_n). \end{multline*}

Taking $\mu$ of both sides and summing over $i$ gives

\begin{align*}
\sum_{i=1}^n\mu(\sigma(x_0,\dots,x_{i-1},t,x_{i+1},\dots,x_n))&=\sum_{i=1}^n\mu(\sigma(x_0,\dots,x_{i-1},s,x_{i+1},\dots,x_n))\\
&+\sum_{i=1}^n\mu(\sigma(s,\dots,x_{i-1},t,x_{i+1},\dots,x_n)).
\end{align*}

However, the first term on the right is just $\mu(\sigma(x_0,\dots,x_n))$ and
the second term is $\mu(\sigma(t,x_1,\dots,x_n))$. Combining this with
\eqref{6.3}
gives us \eqref{6.2}, finishing the proof of $\delta f=0$.
\end{proof}

\begin{remark}\label{remCurvature}
  The proof that $\partial f=0$ strongly used the fact that $X$ has constant
  curvature. In the case where $X$ is an oriented Riemannian surface of
  variable curvature, totally geodesic
  $n$ simplices don't generally exist, although geodesic triangles
  $\sigma(x_0,x_1,x_2)$ are well defined for $(x_0,x_1,x_2)\in
  U^3_{\alpha}$. In this case, the proof above shows that $\delta f=0$. More
  generally, for an $n$-dimensional connected oriented Riemannian manifold,
  using the
  order of a tuple $(x_0,\dots, x_n)$ one can iteratively form convex
  combinations and in this way assign an oriented $n$-simplex to
  $(x_0,\dots,x_n)$ and then define the volume cocycle as above (if $\alpha$
  is small enough).

  Using a chain map to simplicial cohomology which evaluates at the vertices'
  points, it is easy to check that these cocycles represent a generator of the
  cohomology in degree $n$ (which by the results of Section \ref{Section9} is
  exactly $1$-dimensional).
\end{remark}

\section{Cohomology}\label{Section7}

Traditional cohomology theories on general spaces are typically defined in
terms of limits as in \v Cech theory, with nerves of coverings. However, an
algorithmic approach suggests a development via a scaled theory, at a given
scale $\alpha>0$. Then, as $\alpha\to 0$ one recovers the classical setting. A
closely related point of view is that of persistent homology, see
Edelsbrunner, Letscher, and Zomorodian \cite{9}, Zomorodian and Carlsson
\cite{25}, and Carlsson \cite{26}.

We give a setting for such a scaled theory, with a fixed scaling parameter
$\alpha>0$.

Let $X$ be a separable, complete metric space with metric $d$, and $\alpha>0$
a ``scale''. We will define a (generally infinite) simplicial complex
$C_{X,\alpha}$ associated to $(X,d,\alpha)$. Toward that end let $X^{\ell+1}$,
for $\ell\geq 0$, be the $(\ell+1)$-fold Cartesian product, with metric still
denoted by $d$, $d\colon X^{\ell+1}\times X^{\ell+1}\to\mathbb{R}$ where
$d(x,y)=\max_{i=0,\dots,\ell}d(x_i,y_i)$. As in Section \ref{Section4}, let
$$
U^{\ell+1}_{\alpha}(X)=U^{\ell+1}_{\alpha}=\{x\in X^{\ell+1}:
d(x,D_{\ell+1})\leq\alpha\}
$$
where $D_{\ell+1}\subset X^{\ell+1}$ is the diagonal, so
$D_{\ell+1}=\{(t,\dots,t)\ \ell+1\ \text{times}\}$. Then let
$C^\ell_{X,\alpha}=U^{\ell+1}_{\alpha}$. This has the
structure of a simplicial complex whose $\ell$-simplices consist of points of
$U^{\ell+1}_{\alpha}$. This is well defined since if $x\in
U^{\ell+1}_{\alpha}$, then $y=(x_0,\dots,\hat x_i,\dots,x_{\ell})\in
U^{\ell}_{\alpha}$, for each $i=0,\dots,\ell$. We will write $\alpha=\infty$
to mean that $U^{\ell}_{\alpha}=X^{\ell}$. Following e.g.~Munkres \cite{17},
there is a well-defined cohomology theory, simplicial cohomology, for this
simplicial complex, with cohomology vector spaces (always over $\mathbb{R}$),
denoted by $H^{\ell}_{\alpha}(X)$. We especially note that $C_{X,\alpha}$ is
not necessarily a finite simplicial complex. For example, if $X$ is an open
non-empty subset of Euclidean space, the vertices of $C_{X,\alpha}$ are the
points of $X$ and of course infinite in number. The complex $C_{X,\alpha}$
will be called the simplicial complex at scale $\alpha$ associated to
$X$.

\begin{example}
  $X$ is finite. Fix $\alpha>0$. In this case, for each $\ell$, the set of
  $\ell$-simplices is finite, the $\ell$-chains form a finite dimensional
  vector space and the $\alpha$-cohomology groups (i.e. vector spaces)
  $H^{\ell}_{\alpha}(X)$ are all finite dimensional. One can check that for
  $\alpha=\infty$, one has dim$H^0_{\alpha}(X)=1$ and $H^i_{\alpha}(X)$ are
  trivial for all $i>0$. Moreover, for $\alpha$ sufficiently small
  ($\alpha<\min\{d(x,y): x,y\in X, \ x\neq y\}$)
  dim$H^0_{\alpha}(X)=$cardinality of $X$, with $H^i_{\alpha}(X)=0$ for all
  $i>0$. For intermediate $\alpha$, the $\alpha$-cohomology can be rich in
  higher dimensions, but $C_{X,\alpha}$ is a finite simplicial complex.
\end{example}

\begin{example}
  First let $A\subset \mathbb{R}^2$ be the annulus $A=\{x\in\mathbb{R}^2:
  1\leq\|x\|\leq 2\}$. Form $A^{\ast}$ by deleting the finite set of points
  with rational coordinates $(p/q,r/s)$, with $|q|,|s|\leq 10^{10}$. Then one may check
  that for $\alpha>4$, $H^{\ell}_{\alpha}(A^{\ast})$ has the cohomology of a
  point, for certain intermediate values of $\alpha$,
  $H^{\ell}_{\alpha}(A^{\ast})=H^{\ell}_{\alpha}(A)$, and for $\alpha$ small
  enough $H^{\ell}_{\alpha}(A^{\ast})$ has enormous dimension. Thus the scale
  is crucial to see the features of $A^{\ast}$ clearly.
\end{example}
Returning to the case of general $X$, note that if $0<\beta<\alpha$ one has a
natural inclusion $J\colon U^{\ell}_{\beta}\to U^{\ell}_{\alpha}$, $J\colon C_{X,\beta}\to
C_{X,\alpha}$ and the restriction $J^{\ast}\colon L^2_a(U^{\ell}_{\alpha})\to
L^2_a(U^{\ell}_{\beta})$ commuting with $\delta$ (a chain map).

Now assume $X$ is compact. For fixed scale $\alpha$, consider the covering
$\{B_{\alpha}(x):x\in X\}$, where $B_{\alpha}(x)$ is the ball
$B_{\alpha}(x)=\{y\in X: d(x,y)<\alpha\}$, and the nerve of the covering is
$C_{X,\alpha}$, giving the ``\v Cech construction at scale $\alpha$''. Thus from
\v Cech cohomology theory, we see that the limit as $\alpha\to 0$ of
$H^{\ell}_{\alpha}(X)=H^{\ell}(X)=H^{\ell}_{\text{\v Cech}}(X)$ is the $\ell$-th \v Cech
cohomology group of $X$.

The next observation is to note that our construction of the scaled simplicial
complex $C_{X,\alpha}$ of $X$ follows the same path as Alexander-Spanier
theory (see Spanier \cite{21}). Thus the scaled cohomology groups
$H^{\ell}_{\alpha}(X)$ will have the direct limit as $\alpha\to 0$ which maps
to the Alexander-Spanier group $H^{\ell}_{\text{Alex-Sp}}(X)$ (and in many cases will
be isomorphic). Thus
$H^{\ell}(X)=H^{\ell}_{\text{Alex-Sp}}(X)=H^{\ell}_{\text{\v Cech}}(X)$. In fact in much of the
literature this is recognized by the use of the term Alexander-Spanier-\v Cech
cohomology. What we have done is describe a finite scale version of the
classical cohomology.

Now that we have defined the scale $\alpha$ cohomology groups,
$H^{\ell}_{\alpha}(X)$ for a metric space $X$, our Hodge theory suggests this
modification. From Theorem \ref{Theorem3}, we have considered instead of
arbitrary cochains
(i.e. arbitrary functions on $U^{\ell+1}_{\alpha}$ which give the definition
here of $H^{\ell}_{\alpha}(X)$), cochains defined by $L^2$-functions on
$U^{\ell+1}_{\alpha}$. Thus we have constructed cohomology groups at scale
$\alpha$ from $L^2$-functions on $U^{\ell+1}_{\alpha}$,
$H^{\ell}_{\alpha,L^2}(X)$, when $\alpha>0$, and $X$ is a metric space
equipped with Borel probability measure.

\begin{question}[Cohomology Identification Problem (CIP)]\label{qCIP}
   To what extent are $H^{\ell}_{L^2,\alpha}(X)$ and
  $H^{\ell}_{\alpha}(X)$ isomorphic?
\end{question}

This is important via Theorem \ref{Theorem3} which asserts that
$H^{\ell}_{\alpha,L^2}(X)\to \Harm_{\alpha}^{\ell}(X)$ is an
isomorphism, in case $H^{\ell}_{\alpha,L^2}(X)$ is finite dimensional.

One may replace $L^2$-functions in the construction of the $\alpha$-scale
cohomology theory by continuous functions. As in the $L^2$-theory, this gives
rise to cohomology groups $H^{\ell}_{\alpha,cont}(X)$. Analogous to CIP we
have the simpler question: To what extent is the natural map
$H^{\ell}_{\alpha,cont}(X)\to H^{\ell}_{\alpha}(X)$ an isomorphism?

We will give answers to these questions for special $X$ in Section
\ref{Section9}.

Note that in the case $X$ is finite, or $\alpha=\infty$, we have an
affirmative answer to this question, as well as CIP (see Sections \ref{Section2} and \ref{Section3}).

\begin{proposition}\label{PropositionA}
  There is a natural injective linear map
  $$\Harm^{\ell}_{cont,\alpha}(X)\to
  H^{\ell}_{cont,\alpha}(X).$$
\end{proposition}

\begin{proof}
  The inclusion, which is injective
$$
J\colon
\Img_{cont,\alpha}\delta\oplus\Harm^{\ell}_{cont,\alpha}(X)\to\Ker_{cont,\alpha}
$$
induces an injection
$$
J^{\ast}\colon \Harm^{\ell}_{cont,\alpha}(X)=\frac{\Img_{cont,\alpha}\delta\oplus\Harm^{\ell}_{cont,\alpha}(X)}{\Img_{cont,\alpha}\delta}\to
\frac{\Ker_{cont,\alpha}}{\Img_{cont,\alpha}}=H^{\ell}_{cont,\alpha}(X)
$$
and the proposition follows.
\end{proof}

\section{Continuous Hodge theory on the neighborhood of the
  diagonal}\label{Section8}

As in the last section, $(X,d)$ will denote a compact metric space equipped
with a Borel probability measure $\mu$. For topological reasons (see Section
\ref{Section6}) it would be nice to have a Hodge decomposition for continuous
functions on
$U^{\ell+1}_{\alpha}$, analogous to the continuous theory on the whole space
(Section \ref{Section4}). We will use the following
notation. $C^{\ell+1}_{\alpha}$ will
denote the continuous alternating real valued functions on
$U^{\ell+1}_{\alpha}$, $\Ker_{\alpha,cont}\Delta_{\ell}$ will denote the
functions in $C^{\ell+1}_{\alpha}$ that are harmonic, and
$\Ker_{\alpha,cont}\delta_{\ell}$ will denote those elements of
$C^{\ell+1}_{\alpha}$ that are closed. Also, $H^{\ell}_{\alpha,cont}(X)$ will
denote the quotient space (cohomology space)
$\Ker_{\alpha,cont}\delta_{\ell}/\delta(C^{\ell}_{\alpha})$. We raise the
following question, analogous to Theorem \ref{Theorem3}.

\begin{question}[Continuous Hodge Decomposition]\label{qCHD} Under what conditions on
  $X$ and
$\alpha>0$ is it true that there is the following orthogonal (with respect to
the $L^2$-inner product) direct sum decomposition
$$
C^{\ell+1}_{\alpha}=
\delta(C^{\ell}_{\alpha})\oplus \partial(C^{\ell+2}_{\alpha})\oplus
\Ker_{\alpha,cont}\Delta_{\ell}
$$
where $\Ker_{cont,\alpha}\Delta_{\ell}$ is isomorphic to
$H^{\ell}_{\alpha,cont}(X)$, with every element in $H^{\ell}_{\alpha,cont}(X)$
having a unique representative in $\Ker_{\alpha,cont}\Delta_{\ell}$?
\end{question}

There is a related analytical problem that is analogous to elliptic regularity
for partial differential equations, and in fact elliptic regularity features
prominently in classical Hodge theory.

\begin{question}[The Poisson Regularity Problem] \label{qPRP} For $\alpha>0$, and $\ell>0$, suppose that
$\Delta f=g$ where $g\in C^{\ell+1}_{\alpha}$ and $f\in
L^2_a(U^{\ell+1}_{\alpha})$. Under what conditions on $(X,d,\mu)$ is $f$
continuous?
\end{question}

\begin{theorem}
  An affirmative answer to the Poisson Regularity problem, together with
  closed image $\delta(L^2_a(U^{\ell}_{\alpha}))$ implies an affirmative
  solution to the continuous Hodge decomposition question.
\end{theorem}

\begin{proof}
  Assume that the Poisson regularity property holds, and let $f\in
  C^{\ell+1}_{\alpha}$. From Theorem \ref{Theorem3} we have the $L^2$-Hodge
  decomposition
$$
f=\delta f_1 +\partial f_2 +f_3
$$
where $f_1\in L^2_a(U^{\ell}_{\alpha})$, $f_2\in L^2_a(U^{\ell+2}_{\alpha})$
and $f_3\in L^2_a(U^{\ell+1}_{\alpha})$ with $\Delta f_3=0$. It suffices to
show that $f_1$ and $f_2$ can be taken to be continuous, and $f_3$ is
continuous. Since $\Delta f_3=0$ is continuous, $f_3$ is continuous by Poisson
regularity. We will show that $\partial f_2=\partial(\delta h_2)$ where
$\delta h_2$ is continuous (and thus $f_2$ can be taken to be
continuous). Recall (corollary of the Hodge Lemma in Section \ref{Section2})
that the following maps are isomorphisms
$$
\delta\colon \partial (L^2_a(U^{\ell+2}_{\alpha}))\to
\delta(L^2_a(U^{\ell+1}_{\alpha}))\ \text{and} \ \partial\colon
\delta(L^2_a(U^{\ell}_{\alpha}))\to \partial(L^2_a(U^{\ell+1}_{\alpha}))
$$
for all $\ell\geq 0$. Thus
$$
\partial f_2=\partial(\delta h_2)\ \text{for some}\ h_2\in
L^2_a(U^{\ell+1}_{\alpha}).
$$
Now,
\begin{equation}\label{star2}
\Delta
(\delta(h_2))=\delta(\partial(\delta(h_2)))+\partial(\delta(\delta(h_2)))=\delta(\partial(\delta(h_2)))=\delta(\partial(f_2))
\end{equation}
since $\delta^2=0$. However, from the decomposition for $f$ we have, since
$\delta f_3=0$
$$
\delta f=\delta(\partial f_2)
$$
and since $f$ is continuous $\delta f$ is continuous, and therefore
$\delta(\partial f_2)$ is continuous. It then follows from Poisson regularity
and~\eqref{star2} that $\delta h_2$ is continuous as to be shown. A dual argument shows
that $\delta f_1=\delta(\partial h_1)$ where $\partial h_1$ is continuous,
completing the proof.
\end{proof}

Notice that a somewhat weaker result than Poisson regularity would imply that
$f_3$ above is continuous, namely regularity of harmonic functions.

\begin{question}[Harmonic Regularity Problem] \label{qHRP} For $\alpha>0$, and $\ell>0$,
  suppose that $\Delta f=0$ where $f\in L^2_a(U^{\ell+1}_{\alpha})$. What
  conditions on $(X,d,\mu)$ would imply $f$ is continuous?
\end{question}

Under some additional conditions on the measure, we have answered this for
$\ell=0$ (see Section \ref{Section5}) and can do so for $\ell=1$, which we now
consider.

We assume in addition that the inclusion of continuous functions into
$L^2$-functions induces an epimorphism of the associated
Alexander-Spanier-\v Cech cohomology groups, i.e.~that every cohomology class in
the $L^2$-theory has a continuous representative. In Section \ref{Section9} we
will see that this is often the case.

Let now $f\in L^2_a(U^2_\alpha)$ be harmonic. Let $g$ be a continuous function
in the same cohomology class. Then there is $x\in L^2_a(U^1_\alpha)$ such that
$f=g+dx$. As $\delta^* f=0$ it follows that $\delta^*dx=-\delta^*g$ is
continuous. If the Poisson regularity property in degree zero holds (compare
Proposition \ref{PoissonRegularity} of Section \ref{Section5})  then $x$ is
continuous and therefore also $f=g+dx$ is continuous.

Thus  we
have the following
proposition.
\begin{proposition}
  Assume that $\mu(S_x\cap A)$
  are continuous for $x\in X$ and all $A$ measurable. Assume that every
  cohomology class of degree $1$ has a continuous representative. If $f$ is an
  $\alpha$-harmonic $1$-form in $L^2_a(U^2_{\alpha})$, then $f$ is continuous.
\end{proposition}

As in Section \ref{Section5}, if $\mu$ is Borel regular, it suffices that the
hypotheses
hold for all $A$ closed (or all $A$ open).

\section{Finite dimensional cohomology}\label{Section9}

In this section, we will establish conditions on $X$ and $\alpha>0$ that imply
that the $\alpha$ cohomology is finite dimensional. In particular, in
the case of the $L^2$-$\alpha$ cohomology, they imply that the image
of $\delta$ is closed, and that
 Hodge theory for $X$ at scale $\alpha$ holds. Along the way, we will
compute the $\alpha$-cohomology in terms of ordinary \v Cech cohomology of a
covering and that the different variants of our Alexander-Spanier-\v Cech
cohomology at fixed scale ($L^2$, continuous,\ldots) are all isomorphic. We
then show that the
important class of metric spaces, Riemannian manifolds satisfy these
conditions for $\alpha$ small. In particular, in this case the
$\alpha$-cohomology will be isomorphic to ordinary cohomology with
$\mathbb{R}$-coefficients.

{Note that in \cite[Section 4]{MR1377309}, a Rips version of the
$L^2$-Alexander-Spanier complex a finite scale is introduced which is similar
to ours. It is then sketched how, for sufficiently small scales on a manifold
or a simplicial complex, its cohomology should be computable in terms of the
$L^2$-simplicial or $L^2$-de Rham cohomology, without giving detailed
arguments. These results are rather similar to our results. The fact that we
work with the $\alpha$-neighborhood of the diagonal causes some additional
difficulties we have to overcome.}

Throughout this section, $(X,d)$ will denote a compact metric space, $\mu$ a
Borel probability measure on $X$ such that $\mu(U)>0$ for all nonempty open
sets $U\subset X$, and $\alpha>0$. As before $U^{\ell}_{\alpha}$ will denote
the closed $\alpha$-neighborhood of the diagonal in $X^{\ell}$. We will denote
by $F_a(U^{\ell}_{\alpha})$ the space of all alternating real valued functions
on $U^{\ell}_{\alpha}$, by $C_a(U^{\ell}_{\alpha})$ the continuous alternating
real valued functions on $U^{\ell}_{\alpha}$, and by $L^p_a(U^{\ell}_{\alpha})$
the $L^p$ alternating real valued functions on $U^{\ell}_{\alpha}$ for $p\geq
1$ (in particular, the case $p=2$ was discussed in the preceding sections). If
$X$ is a smooth
Riemannian manifold, $C^{\infty}_a(U^{\ell}_{\alpha})$ will be the smooth
alternating real valued functions on $U^{\ell}_{\alpha}$. We will be
interested in the following cochain complexes:

\begin{equation*}
\begin{CD}
  0 @>>> L^p_a(X) @>{\delta_0} >> L^p_a(U^2_{\alpha}) @>{\delta_1} >> \cdots
  @>{\delta_{\ell-1} }>> L^p_a(U^{\ell+1}_{\alpha}) @>{\delta_{\ell}} >>
  \cdots
\end{CD}
\end{equation*}

\begin{equation*}
  \begin{CD}
    0 @>>> C_a(X) @>{\delta_0} >> C_a(U^2_{\alpha}) @>{\delta_1} >> \cdots
    @>{\delta_{\ell-1} }>> C_a(U^{\ell+1}_{\alpha}) @>{\delta_{\ell}} >>
    \cdots
  \end{CD}
\end{equation*}

\begin{equation*}
\begin{CD}
  0 @>>> F_a(X) @>{\delta_0} >> F_a(U^2_{\alpha}) @>{\delta_1} >> \cdots
  @>{\delta_{\ell-1} }>> F_a(U^{\ell+1}_{\alpha}) @>{\delta_{\ell}} >> \cdots
\end{CD}
\end{equation*}

And if $X$ is a smooth Riemannian manifold,
\begin{equation*}
\begin{CD}
  0 @>>> C^{\infty}_a(X) @>{\delta_0} >> C^{\infty}_a(U^2_{\alpha})
  @>{\delta_1} >> \cdots @>{\delta_{\ell-1} }>>
  C^{\infty}_a(U^{\ell+1}_{\alpha}) @>{\delta_{\ell}} >> \cdots
\end{CD}
\end{equation*}

The corresponding cohomology spaces
$\Ker\delta_{\ell}/\Img\delta_{\ell-1}$ will be denoted by
$H^{\ell}_{\alpha,L^p}(X)$, or briefly $H^{\ell}_{\alpha,L^p}$,
$H^{\ell}_{\alpha,cont}$, $H^{\ell}_{\alpha}$ and $H^{\ell}_{\alpha,smooth}$
respectively. The proof of finite dimensionality of these spaces, under
certain conditions, involves the use of bicomplexes, some facts about which we
collect here.

A bicomplex $C^{\ast,\ast}$ will be a rectangular array of vector spaces
$C^{j,k}$, $j,k\geq 0$, and linear maps (coboundary operators)
$c_{j,k}\colon C^{j,k}\to C^{j+1,k}$, and $d_{j,k}\colon C^{j,k}\to C^{j,k+1}$ such that
the rows and columns are chain complexes, that is $c_{j+1,k} c_{j,k}=0$,
$d_{j,k+1} d_{j,k}=0$, and $c_{j,k+1}d_{j,k}=d_{j+1,k}c_{j,k}$. Given such a
bicomplex, we associate the total complex $E^{\ast}$, a chain complex

\begin{equation*}
\begin{CD}
  0 @>>> E^0 @>{D_0} >>E^1 @>{D_1} >> \cdots @>{D_{\ell-1} }>> E^{\ell}
  @>{D_{\ell}} >> \cdots
\end{CD}
\end{equation*}
where $E^{\ell}=\bigoplus_{j+k=\ell}C^{j,k}$ and where on each term $C^{j,k}$
in $E^{\ell}$, $D_{\ell}=c_{j,k}+(-1)^kd_{j,k}$. Using commutativity of $c$ and
$d$, one can easily check that $D_{\ell+1}D_{\ell}=0$, and thus the total
complex is a chain complex. We recall a couple of definitions from homological
algebra. If $E^{\ast}$ and $F^{\ast}$ are cochain complexes of vector spaces
with coboundary operators $e$ and $f$ respectively, then a chain map
$g\colon E^{\ast}\to F^{\ast}$ is a collection of linear maps $g_j\colon E^j\to F^j$
that
commute with $e$ and $f$. A chain map induces a map on cohomology. A cochain
complex $E^{\ast}$ is said to be exact at the $k$th term if the kernel of
$e_k\colon E_k\to E_{k+1}$ is equal to the image of $e_{k-1}\colon E_{k-1}\to
_k$. Thus
the cohomology at that term is zero. $E^{\ast}$ is defined to be exact if it
is exact at each term. A chain contraction $h\colon E^{\ast}\to E^{\ast}$ is a
family of linear maps $h_j\colon E^j\to E^{j-1}$ such that
$e_{j-1}h_j+h_{j+1}e_j=\text{Id}$. The existence of a chain contraction on
$E^{\ast}$ implies that $E^{\ast}$ is exact. The following fact from
homological algebra is fundamental in proving finite dimensionality of our
cohomology spaces.
\begin{lemma}
  Suppose that $C^{\ast,\ast}$ is a bicomplex as above, and $E^{\ast}$ is the
  associated total complex. Suppose that we augment the bicomplex with a
  column on the left which is a chain complex $C^{-1,\ast}$,

  \begin{equation*}
    \begin{CD}
      C^{-1,0} @>{d_{-1,0}} >>C^{-1,1} @>{d_{-1,1} }>> \cdots
      @>{d_{-1,\ell-1} }>> C^{-1,\ell} @>{d_{-1,\ell}} >> \cdots
    \end{CD}
  \end{equation*}
  and with linear maps $c_{-1,k}\colon C^{-1,k}\to C^{0,k}$, such that the
  augmented rows

  \begin{equation*}
    \begin{CD}
      0 @>>> C^{-1,k} @>{c_{-1,k}} >>C^{0,k} @>{c_{0,k}} >> \cdots
      @>{c_{\ell-1,k} }>> C^{\ell,k} @>{c_{\ell,k}} >> \cdots
    \end{CD}
  \end{equation*}
  are chain complexes with $d_{0,k}c_{-1,k}=c_{-1,k+1}d_{-1,k}$. Then, the
  maps $c_{-1,k}$ induce a chain map $c_{-1,\ast}\colon C^{-1,\ast}\to
  E^{\ast}$. Furthermore, if the first $K$ rows of the augmented complex are
  exact, then $c_{-1,\ast}$ induces an isomorphism on the homology of the
  complexes $c^{\ast}_{-1,\ast}\colon H^k(C^{-1,\ast})\to H^k(E^{\ast})$ for $k\leq
  K$ and an injection for $k=K+1$. In fact, one only needs exactness of the
  first $K$ rows up to the $K$th term $C^{K,j}$.
\end{lemma}

A simple proof of this is given in Bott and Tu \cite[pages 95--97]{BottTu}, in
the
case of the \v Cech-de Rham complex, but the the proof generalizes to the
abstract setting. Of course, if we augmented the bicomplex with a row
$C^{\ast,-1}$ with the same properties, the conclusions would hold. In fact,
we will show the cohomologies of two chain complexes are isomorphic by
augmenting a bicomplex as above with one such row and one such column.

\begin{corollary}
  Suppose that $C^{\ast,\ast}$ is a bicomplex as in the Lemma, and that
  $C^{\ast,\ast}$ is augmented with a column $C^{-1,\ast}$ as in the Lemma,
  and with a row $C^{\ast,-1}$ that is also a chain complex with coboundary operators
  $c_{j,-1}\colon C^{j,-1}\to C^{j+1,-1}$ and linear maps $d_{j,-1}\colon C^{j,-1}\to
  C^{j,0}$ such that the augmented columns
$$
\begin{CD}
  0 @>>> C^{j,-1} @>{d_{j,-1}} >>C^{0,k} @>{d_{j,-1}} >> \cdots
  @>{d_{j,\ell-1} }>> C^{j,\ell} @>{d_{j,\ell}} >> \cdots
\end{CD}
$$
are chain complexes, and $c_{j,0}d_{j,-1}=d_{j+1,-1}c_{j,-1}$. Then, if the
first $K$ rows are exact and the first $K+1$ columns are exact, up to the $K+1$
term, it follows that the cohomology $H^{\ell}(C^{-1,\ast})$ of $C^{-1,\ast}$
and $H^{\ell}(C^{\ast,-1})$ of $C^{\ast,-1}$ are isomorphic for $0\leq K$,
and $H^{K+1}(C^{-1,\ast})$ is isomorphic to a subspace of
$H^{K+1}(C^{\ast,-1})$.
\end{corollary}

\begin{proof}
  This follows immediately from the lemma, as the cohomology up to order $K$
  of both $C^{-1,\ast}$ and $C^{\ast,-1}$ are isomorphic to the cohomology of
  the total complex. Also, $H^{K+1}(C^{-1,\ast})$ is isomorphic to a subspace
  of $H^{K+1}(E^{\ast})$ which is isomorphic to $H^{K+1}(C^{\ast,-1})$.
\end{proof}

\begin{remark}
  If all of the spaces $C^{j,k}$ in the Lemma and Corollary are Banach spaces,
  and the coboundaries, $c_{j,k}$ and $d_{j,k}$ are bounded, then the
  isomorphisms of cohomology can be shown to be topological isomorphisms,
  where the topologies on the cohomology spaces are induced by the quotient
  semi-norms.
\end{remark}

Let $\{V_i, i\in S\}$ be a finite covering of $X$ by Borel sets (usually taken
to be balls). We construct the corresponding \v Cech-$L^p$-Alexander bicomplex
at scale $\alpha$ as follows.

\begin{equation*}
C^{k,\ell}=\bigoplus_{I\in S^{k+1}}L^p_a(U^{\ell+1}_{\alpha}\cap
V^{\ell+1}_I)\ \ \text{for}\ k,\ell\geq 0
\end{equation*}
where we use the abbreviation $V_I=V_{i_0,\dots,i_k}=\bigcap_{j=0}^k
V_{i_j}$. The vertical coboundary $d_{k,\ell}$ is just the usual coboundary
$\delta_{\ell}$ as in Section \ref{Section4}, acting on each
$L^p_a(U^{\ell+1}_{\alpha}\cap
V_{I^{\ell+1}})$. The horizontal coboundary $c_{k,\ell}$ is the ``\v Cech
differential''. More explicitly, if $f\in C^{k,\ell}$, then it has components $f_I$
which are functions on $U^{\ell+1}_{\alpha}\cap V^{\ell+1}_I$ for each $(k+1)$-tuple $I$, and for any $k+2$ tuple $J=(j_0,\dots,j_{k+1})$, $cf$ is defined on
$U^{\ell+1}_{\alpha}\cap V^{\ell+1}_J$ by
$$
(c_{k,\ell}f)_J=\sum_{i=0}^{k+1}(-1)^if_{j_0,\dots,\hat j_i,\dots,j_{k+1}}\ \
\text{restricted to}\ \ V^{\ell+1}_J.
$$
It is not hard to check that the coboundaries commute $c\delta=\delta c$. We
augment the complex on the left with the column (chain complex)
$C^{-1,\ell}=L^p_a(U^{\ell+1}_{\alpha})$ with horizontal map $c_{-1,\ell}$
equal to restriction on each $V_i$ and vertical map the usual coboundary. We
augment the complex on the bottom with the chain complex $C^{\ast,-1}$ which
is the \v Cech complex of the cover $\{V_i\}$. That is an element $f\in C^{k,-1}$
is a function that assigns to each $V_I$ a real number or equivalently
$C^{k,-1}=\bigoplus_{I\in S^{k+1}}\mathbb{R} V_I$. The vertical maps are just
inclusions into $C^{\ast,0}$, and the horizontal maps are the \v Cech
differential as defined above.
\begin{remark}
  We can similarly define the \v Cech-Alexander bicomplex, the \v Cech-Continuous
  Alexander bicomplex and the \v Cech-Smooth Alexander bicomplex (in case $X$ is
  a smooth Riemannian manifold) by replacing $L^p_a$ everywhere in the above
  complex with $F_a$, $C_a$ and $C^{\infty}_a$ respectively.
\end{remark}
\begin{remark}
   The cohomology spaces of $C^{\ast,-1}$ are finite
  dimensional since the cover $\{V_i\}$ is finite. This is called the \v Cech
  cohomology of the cover, and is the same as the simplicial cohomology of the
  simplicial complex that is the nerve of the cover $\{V_i\}$.
\end{remark}
We will use the above complex to show, under some conditions, that
$H^{\ell}_{\alpha,L^p}$, $H^{\ell}_{\alpha}$ and $H^{\ell}_{\alpha,cont}$ are
isomorphic to the \v Cech cohomology of an appropriate finite open cover of $X$
and thus finite dimensional.

\begin{theorem}
  \label{Theorem9.1} Let $\{V_i\}_{i\in S}$ be a finite cover of $X$ by
  Borel sets as above, and assume that $\{V^{K+1}_i\}_{i\in S}$ is a cover for
  $U^{K+1}_{\alpha}$ for some $K\geq 0$. Assume also that the first $K+1$
  columns of the corresponding \v Cech-$L^p$-Alexander complex are exact up to
  the $K+1$ term. Then $H^{\ell}_{\alpha,L^p}$ is isomorphic to
  $H^{\ell}(C^{\ast,-1})$ for $\ell\leq K$ and is thus finite
  dimensional. Also $H^{K+1}_{\alpha,L^p}$ is isomorphic to a subspace of
  $H^{K+1}(C^{\ast,-1})$. If $\{V_i\}_{i\in S}$ is an open cover, then the
  same conclusion holds for $H^{\ell}_{\alpha}$, $H^{\ell}_{\alpha,cont}$ and
  $H^{\ell}_{\alpha,smooth}$ (in case $X$ is a smooth Riemannian manifold),
  and hence all are isomorphic to each other. Those isomorphisms are induced
  by the natural inclusion maps of smooth functions into continuous functions
  into $L^q$-functions into $L^p$-functions ($q\ge p$) into arbitrary real
  valued functions.
\end{theorem}

\begin{proof}
  In light of the corollary above, it suffices to show that the first $K$ rows
  of the bicomplex are exact. Indeed, we are computing the sheaf cohomology
  of $U_\alpha^{k+1}$ for a flabby sheaf (the sheaf of smooth or continuous or
  $L^p$ or
  arbitrary functions) which vanishes. We write out the details:  Note that for $\ell\leq K$, $\{V^{\ell+1}_i\}$
  covers $U^{\ell+1}_{\alpha}$ and therefore
  $c_{-1,\ell}\colon L^p_a(U^{\ell+1}_{\alpha})\to \bigoplus_{i\in S}
  L^p_a(U^{\ell+1}_{\alpha}\cap V^{\ell+1}_i)$ is injective (as $c_{-1,\ell}$
  is restriction), and therefore we have exactness at the first term. In
  general, we construct a chain contraction $h$ on the $\ell$th row. Let
  $\{\phi_i\}$ be a measurable partition of unity for $U^{\ell+1}_{\alpha}$
  subordinate to the cover $\{U^{\ell+1}_{\alpha}\cap V^{\ell+1}_i\}$ (thus
  $\operatorname{support}\phi_i\subset U^{\ell+1}_{\alpha}\cap V^{\ell+1}_i$ and
  $\sum_i\phi_i(x)=1$ for all $x$). Then define
$$h\colon \bigoplus_{I\in S^{k+1}} L^p_a(U^{\ell+1}_{\alpha}\cap V^{\ell+1}_I) \to \bigoplus_{I\in S^{k}} L^p_a(U^{\ell+1}_{\alpha}\cap V^{\ell+1}_I)$$
for each $k$ by $(hf)_{i_0,\dots,i_{k-1}}=\sum_{j\in S}\phi_j
f_{j,i_0,\dots,i_{k-1}}$. We show that $h$ is a chain contraction, that is
$ch+hc=\text{Id}$:
$$
(c(hf))_{i_0,\dots,i_k-1}=\sum_{n=0}^{k-1}(-1)^n(h f)_{i_0,\dots,\hat
  i_n,\dots,i_{k-1}}=\sum_{j,n}(-1)^n\phi_j f_{j,i_0,\dots,\hat
  i_n,\dots,i_{k-1}}.
$$
Now,

\begin{align*}
 (h(cf))_{i_0,\dots,i_k-1}&=\sum_{j\in S}\phi_j(cf)_{j,i_0,\dots,i_{k-1}}\\
&=\sum_j\phi_j(f_{i_0,\dots,i_{k-1}}-\sum_n^{k-1} (-1)^n f_{j,i_0,\dots,\hat i_n,\dots,i_{k-1}})\\
&=f_{i_0,\dots,i_{k-1}}-(c(hf))_{i_0,\dots,i_{k-1}}.
\end{align*}

Thus $h$ is a chain contraction for the $\ell$th row, proving exactness (note
that exactness follows, since if $cf=0$ then from above $c(hf)=f$). If
$\{V_i\}$ is an open cover, then the partition of unity $\{\phi_i\}$ can be
chosen to be continuous, or even smooth in case $X$ is a smooth Riemannian
manifold. Then $h$ as defined above is a chain contraction on the
corresponding complexes with $L^p_a$ replaced by $F_a$, $C_a$ or
$C^{\infty}_a$.

Observe that the inclusions $C^\infty\hookrightarrow C^0\hookrightarrow
L^q\hookrightarrow L^p\hookrightarrow F$
(where $F$ stands for arbitrary real valued functions) extend to inclusions of
the augmented bicomplexes, whose restriction to the \v Cech column $C^{*,-1}$ is
the identity. As the identity clearly induces an isomorphism in cohomology,
and the inclusion of this augmented bottom row into the (non-augmented)
bicomplex also does, by naturality the various inclusions of the bicomplexes
induce isomorphisms in cohomology. The same argument applied backwards to the
inclusions of the Alexander-Spanier-\v Cech rows into the bicomplexes shows that
the inclusions of the smaller function spaces into the larger function spaces
induce isomorphisms in $\alpha$-cohomology.

This finishes the proof of the theorem.
\end{proof}

We can use Theorem \ref{Theorem9.1} to prove finite dimensionality of the
cohomologies in general, for $\ell=0$ and $1$.

\begin{theorem}
  \label{Theorem9.2} For any compact $X$ and any $\alpha>0$,
  $H^{\ell}_{\alpha,L^p}$, $H^{\ell}_{\alpha}$, $H^{\ell}_{\alpha,cont}$, and
  $H^{\ell}_{\alpha,smooth}$ ($X$ a smooth manifold) are finite dimensional
  and are isomorphic, for $\ell=0,1$.
\end{theorem}

Let $\{V_i\}$ be a covering of $X$ by open balls of radius $\alpha/3$. Then
the first row ($\ell=0$) of the \v Cech-$L^p$-Alexander Complex is exact from the
proof of Theorem \ref{Theorem9.1} (taking $K=0$). It suffices to show that the columns are
exact. Note that $V^{\ell+1}_I\subset U^{\ell+1}_{\alpha}$ trivially, for each
$\ell$ and $I\in S^{k+1}$ because $\diam(V_I)<\alpha$. For $k$ fixed, and
$I\in S^{k+1}$ we define $g\colon L^p_a(V^{\ell+1}_I)\to L^p_a(V^{\ell}_I)$ by
$$
gf(x_0,\dots,x_{\ell-1})=\frac{1}{\mu(V_I)}\int_{V_I}f(t,x_0,\dots,x_{\ell-1})\,d\mu(t).
$$
We check that $g$ defines a chain contraction:
\begin{align*}
\delta(gf)(x_0,\dots,x_{\ell})&=\sum_i(-1)^i(gf)(x_0,\dots,\hat x_i,\dots,x_{\ell})\\
&=\sum_i(-1)^i\frac{1}{\mu(V_I)}\int_{V_I}f(t,x_0,\dots,\hat
x_i,\dots,x_{\ell})\,d\mu(t).
\end{align*}
But,
\begin{align*}
g&(\delta f)(x_0,\dots,x_{\ell})=\frac{1}{\mu(V_I)}\int_{V_I }\delta f(t,x_0,\dots,x_{\ell})\, d\mu(t)\\
&=\frac{1}{\mu(V_I)}\left(\int_{V_I}f((x_0,\dots,x_{\ell})\, d\mu(t)-\sum_i(-1)^i\int_{V_I}f(t,\dots,\hat x_i,\dots,x_{\ell})\,d\mu)t)\right)\\
&=f(x_0,\dots,x_{\ell})-\delta(gf)(x_0,\dots,x_{\ell}).
\end{align*}

Thus $g$ defines a chain contraction on the $k$th column and the columns are
exact. For the corresponding Alexander, continuous and smooth bicomplexes, a
chain contraction can be defined by fixing for each $V_I$, $I\in S^{k+1}$ a
point $p\in V_I$ and setting
$gf(x_0,\dots,x_{\ell-1})=f(p,x_0,\dots,x_{\ell-1})$. This is easily verified
to be a chain contraction, finishing the proof of the theorem.

Recall that for $x=(x_0,\dots,x_{\ell-1})\in U^{\ell}_{\alpha}$ we define the
slice $S_x=\{t\in X: (t,x_0,\dots,x_{\ell-1})\in U^{\ell+1}_{\alpha}\}$. We
consider the following hypothesis on $X$, $\alpha>0$ and non-negative integer
$K$:

\begin{definition}
  \textbf{\boldmath Hypothesis $(\ast)$}: There exists a $\delta>0$ such that whenever
  $V=\cap_i V_i$ is a non-empty intersection of finitely many open balls of
  radius $\alpha+\delta$, then there is a Borel set $W$ of positive measure
  such that for each $\ell\leq K+1$
$$
W\subset V\cap\left(\bigcap_{x\in U^{\ell}_{\alpha}\cap V^{\ell}}S_x\right).
$$
\end{definition}

\begin{theorem}
  \label{Theorem9.3} Assume that $X$, $\alpha>0$ and $K$ satisfy
  hypothesis $(\ast)$. Then, for $\ell\leq K$, $H^{\ell}_{\alpha,L^p}$,
  $H^{\ell}_{\alpha}$, $H^{\ell}_{\alpha,cont}$, and
  $H^{\ell}_{\alpha,smooth}$ (in the case $X$ is a smooth Riemannian manifold)
  are all finite dimensional, and are isomorphic to the \v Cech cohomology of
  some finite covering of $X$ by open balls of radius
  $\alpha+\delta$. Furthermore, the Hodge theorem for $X$ at scale $\alpha$
  holds (Theorem \ref{Theorem3} of Section \ref{Section4}).
\end{theorem}
\begin{proof}
  Let $\{V_i\}$, $i\in S$ be a finite open cover of $X$ by balls of radius
  $\alpha+\delta$ such that $\{V^{K+1}_i\}$ is a covering for
  $U^{K+1}_{\alpha}$. This can always be done since $U^{K+1}_{\alpha}$ is
  compact. We first consider the case of the \v Cech-$L^p$-Alexander bicomplex
  corresponding to the cover. By Theorem \ref{Theorem9.1}, it suffices to show
  that there is a
  chain contraction of the columns up to the $K$th term. For each $I\in
  S^{k+1}$, and $\ell\leq K+1$ let $W$ be the Borel set of positive measure
  assumed to exist in $(\ast)$ with $V_I$ playing the role of $V$ in
  $(\ast)$. Then we define $g\colon L^p_a(U^{\ell+1}_{\alpha}\cap V^{\ell+1}_I)\to
  L^p_a(U^{\ell}_{\alpha}\cap V^{\ell}_I)$ by
$$
gf(x_0,\dots,x_{\ell-1})=\frac{1}{\mu(W)}\int_W f(t,x_0,\dots,x_{\ell-1})\,
d\mu(t).
$$
The hypothesis $(\ast)$ implies that $g$ is well defined. The proof that $g$
defines a chain contraction on the $k$th column (up to the $K$th term) is
identical to the one in the proof of Theorem \ref{Theorem9.2}. As in the proof
of Theorem \ref{Theorem9.2}, the chain contraction for the case when $L^p_a$
is replaced by $F_a$, $C_a$ and
$C^{\infty}_a$ can be taken to be
$gf(x_0,\dots,x_{\ell-1})=f(p,x_0,\dots,x_{\ell-1})$ for some fixed $p\in
W$. Note that in these cases, we don't require that $\mu(W)>0$, only that
$W\neq\emptyset$.
\end{proof}

\begin{remark}
   If $X$ satisfies certain local conditions as in
  Wilder \cite{28}, then the \v Cech cohomology of the cover, for small $\alpha$,
  is isomorphic to the \v Cech cohomology of $X$.
\end{remark}
Our next goal is to give somewhat readily verifiable conditions on $X$ and
$\alpha$ that will imply $(\ast)$. This involves the notion of midpoint and
radius of a closed set in $X$.

Let $\Lambda\subset X$ be closed. We define the radius $r(\Lambda)$ by
$r(\Lambda)=\inf\{ \beta: \cap_{x\in\Lambda} B_{\beta}(x)\neq
\emptyset\}$ where $B_{\beta}(x)$ denotes the closed ball of radius $\beta$
centered at $x$.
\begin{proposition}
  \label{Proposition9.1}
  $\cap_{x\in\Lambda}B_{r(\Lambda)}(x)\neq\emptyset$. Furthermore, if $p\in
  \cap_{x\in\Lambda}B_{r(\Lambda)}(x)$, then $\Lambda\subset
  B_{r(\Lambda)}(p)$, and if $\Lambda\subset B_{\beta}(q)$ for some
  $q\in\Lambda$, then $r(\Lambda)\leq\beta$.
\end{proposition}

\noindent Such a $p$ is called a midpoint of $\Lambda$.
\begin{proof}
Let $J=\{\beta\in\mathbb{R}\colon
\cap_{x\in\Lambda} B_{\beta}(x)\neq \emptyset\}$. For $\beta\in J$ define
$R_{\beta}= \cap_{x\in\Lambda} B_{\beta}(x)$. Note that if $\beta\in J$ and
$\beta<\beta'$, then $\beta'\in J$, and $R_{\beta}\subset
R_{\beta'}$. $R_{\beta}$ is compact, and therefore $\cap_{\beta\in
  J}R_{\beta}\neq\emptyset$.  Let $p\in \cap_{\beta\in J}R_{\beta}$. Then, for
$x\in\Lambda$, $p\in B_{\beta}(x)$ for all $\beta\in J$ and so
$d(p,x)\leq\beta$. Taking the infimum of this over $\beta\in J$ yields
$d(p,x)\leq r(\Lambda)$ or $p\in R_{r(\Lambda)}$ proving the first assertion
of the proposition. Now, if $x\in\Lambda$ then $p\in B_{r(\Lambda)}(x)$ which
implies $x\in B_{r(\Lambda)}(p)$ and thus $\Lambda\subset
B_{r(\Lambda)}(p)$. Now suppose that $\Lambda\subset B_{\beta}(q)$ for some
$q\in\Lambda$. Then for every $x\in\Lambda$, $q\in B_{\beta}(x)$ and thus
$\cap_{x\in\Lambda} B_{\beta}(x)\neq\emptyset$ which implies $\beta\geq
r(\Lambda)$ finishing the proof.
\end{proof}

We define $\mathcal{K}(X)=\{\Lambda\subset X: \Lambda\ \text{is compact}\}$, and we
endow $\mathcal{K}(X)$ with the Hausdorff metric $D(A,B)=\max\{\sup_{t\in
  B}d(t,A),\sup_{t\in A}d(t, B)\}$. We also define, for
$x=(x_0,\dots,x_{\ell})\in U^{\ell+1}_{\alpha}$, the witness set of $x$ by
$w_{\alpha}(x)=\cap_i B_{\alpha}(x_i)$ (we are suppressing the dependence of
$w_{\alpha}$ on $\ell$). Thus $w_{\alpha}\colon U^{\ell+1}_{\alpha}\to \mathcal{K}(X)$. We have
\begin{theorem}
  \label{Theorem9.4}
Let $X$ be compact, and $\alpha>0$. Suppose that
  $w_{\alpha}\colon U^{\ell+1}_{\alpha}\to \mathcal{K}(X)$ is continuous for $\ell\leq
  K+1$, and suppose there exists $\delta_0>0$ such that whenever
  $\Lambda=\cap_{i=0}^kB_i$ is a finite intersection of closed balls of radius
  $\alpha+\delta$, $\delta\in(0,\delta_0]$ then
  $r(\Lambda)\leq\alpha+\delta$. Then Hypothesis $(\ast)$ holds.
\end{theorem}

\noindent The proof will follow from
\begin{proposition}
  \label{Proposition9.2} Under the hypotheses of Theorem \ref{Theorem9.4},
  given $\epsilon>0$, there exists $\delta>0$, $\delta\leq\delta_0$ such that
  for   all $\beta\in[\alpha,\alpha+\delta]$ we have
  $D(w_{\alpha}(\sigma),w_{\beta}(\sigma))\leq\epsilon$ for all simplices
  $\sigma\in U^{\ell+1}_{\alpha}\subset U^{\ell+1}_{\beta}$.
\end{proposition}

\begin{proof}[Proof of Theorem \ref{Theorem9.4}] Fix $\epsilon<\alpha$, and
  let $\delta>0$ be as
  in Proposition \ref{Proposition9.2}. Let $\{V_i\}$ be a finite collection of
  open balls of
  radius $\alpha+\delta$ such that $\cap_i V_i\neq\emptyset$, and let
  $\{B_i\}$ be the corresponding collection of closed balls of radius
  $\alpha+\delta$. Define $\Lambda$ to be the closure of $\cap_i V_i$ and thus
$$
\Lambda=\overline{\cap_i V_i}\subset\cap_i\overline{V_i}\subset\cap_i B_i.
$$
Let $p$ be a midpoint of $\Lambda$. We will show that
$d(p,w_{\alpha}(\sigma))\leq\epsilon$ for any
$\sigma=(x_0,\dots,x_{\ell+1})\in\Lambda^{\ell+1}$. We have
$$
p\in\cap_{x\in\Lambda}B_{r(\Lambda)}(x)\subset \cap_{i=0}^{\ell+1}
B_{r(\Lambda)}(x_i)=w_{r(\Lambda)}(\sigma)\subset w_{\alpha+\delta}(\sigma)
$$
since $r(\Lambda)\leq\alpha+\delta$. But
$D(w_{\alpha}(\sigma),w_{\alpha+\delta}(\sigma))\leq\epsilon$ from Proposition
\ref{Proposition9.2},
and so $d(p,w_{\alpha}(\sigma))\leq\epsilon$. In particular, there exists
$q\in w_{\alpha}(\sigma)$ with $d(p,q)\leq\epsilon$. Now, if $x\in
B_{\alpha-\epsilon}(p)\cap\Lambda$, then $d(x,q)\leq d(x,p)+d(p,q)\leq
\alpha-\epsilon+\epsilon=\alpha$. Therefore $(x,x_0,\dots,x_{\ell})\in
U^{\ell+2}_{\alpha}$ and so $x\in S_{\sigma}\cap\Lambda$. Thus
$B_{\alpha-\epsilon}(p)\cap\Lambda\subset\cap_{\sigma\in
  U^{\ell+1}\cap\Lambda^{\ell+1}} S_{\sigma}$. Let $B'_s(p)$ denote the open
ball of radius $s$ and let $V=\cap_i V_i$. Then define
$W=B'_{\alpha-\epsilon}(p)\cap V$. Then $W$ is a nonempty open set (since
$p\in\overline{V}$), $\mu(W)>0$ and $W\subset\cap_{\sigma\in
  U^{\ell+1}_{\alpha}\cap V^{\ell+1}}S_{\alpha}$ and Hypothesis $(\ast)$ is
satisfied finishing the proof of Theorem \ref{Theorem9.4}.
\end{proof}

\begin{proof}[Proof of Proposition \ref{Proposition9.2}]
Let $\epsilon>0$. Note that for
  $\beta\geq\alpha$, and $\sigma\in U^{\ell+1}_{\alpha}$,
  $w_{\alpha}(\sigma)\subset w_{\beta}(\sigma)$. It thus suffices to show that
  there exists $\delta>0$ such that
$$
\sup_{x\in w_{\beta}(\sigma)} d(x,w_{\alpha}(\sigma))\leq\epsilon\ \ \text{for
  all}\ \beta\in[\alpha,\alpha+\delta].
$$
Suppose that this is not the case. Then there exists $\beta_j\downarrow\alpha$
and $\sigma_j\in U^{\ell+1}_{\alpha}$ such that
$$
\sup_{x\in w_{\beta_j}(\sigma)} d(x,w_{\alpha}(\sigma_j))>\epsilon
$$
and thus there exists $x_n\in w_{\beta_n}(\sigma_n)$ with
$d(x_n,w_{\alpha}(\sigma_n))\geq\epsilon$.  Let
$\sigma_n=(y^n_0,\dots,y^n_{\ell})$. Thus $d(x_n,y^n_i)\leq\beta_n$ for all
$i$. By compactness, after taking a subsequence, we can assume
$\sigma_n\to\sigma=(y_0,\dots,y_{\ell})$ and $x_n\to x$. Thus
$d(x,y_i)\leq\alpha$ for all $i$ and $\sigma\in U^{\ell+1}_{\alpha}$, and
$x\in w_{\alpha}(\sigma)$. However, by continuity of $w_{\alpha}$,
$w_{\alpha}(\sigma_n)\to w_{\alpha}(\sigma)$ which implies
$d(x,w_{\alpha}(\sigma))\geq\epsilon$ (since
$d(x_n,w_{\alpha}(\sigma_n))\geq\epsilon$) a contradiction, finishing the
proof of the proposition.
\end{proof}

We now turn to the case where $X$ is a compact Riemannian manifold of
dimension $n$, with Riemannian metric $g$, We will always assume that the
metric $d$ on $X$ is induced from $g$. Recall that a set $\Lambda\subset X$ is
strongly convex if given $p,q\in\Lambda$, then the length minimizing geodesic
from $p$ to $q$ is unique, and lies in $\Lambda$. The strong convexity radius
at a point $x\in X$ is defined by $\rho(x)=\sup\{r: B_r(x) :\text{is
  strongly convex}\}$. The strong convexity radius of $X$ is defined as
$\rho(X)=\inf\{\rho(x):x\in X\}$. It is a basic fact of Riemannian
geometry that for $X$ compact, $\rho(X)>0$. Thus for any $x\in X$ and
$r<\rho(X)$, $B_r(x)$ is strongly convex.
\begin{theorem}
  \label{Theorem9.5}
Assume as above that $X$ is a compact Riemannian
  manifold. Let $k>0$ be an upper bound for the sectional
  curvatures of $X$ and let
  $\alpha<\min\{\rho(X),\frac{\pi}{2\sqrt{k}}\}$. Then Hypothesis
  $(\ast)$ holds.
\end{theorem}

\begin{corollary} In the situation of Theorem \ref{Theorem9.5},
 the cohomology groups $H^{\ell}_{\alpha,L^p}$,
  $H^{\ell}_{\alpha}$, $H^{\ell}_{\alpha,cont}$, and
  $H^{\ell}_{\alpha,smooth}$ are finite dimensional and isomorphic to each
  other and to the ordinary cohomology of $X$ with real coefficients (and the
  natural inclusions induce the isomorphisms). Moreover, Hodge theory
  for $X$ at scale $\alpha$ holds.
\end{corollary}

\begin{proof}[Proof of Theorem \ref{Theorem9.5}] From Theorem
  \ref{Theorem9.4}, it suffices to prove the
  following propositions.
  \begin{proposition}
    \label{Proposition9.3} Let
    $\alpha<\min\{\rho(X),\frac{\pi}{2\sqrt{k}}\}$. Then
    $w_{\alpha}\colon U^{\ell+1}_{\alpha}\to \mathcal{K}(X)$ is continuous for $\ell\leq
    K$.
  \end{proposition}

  \begin{proposition}
    \label{Proposition9.4} Let $\delta>0$ such that
    $\alpha+\delta<\min\{\rho(X),\frac{\pi}{2\sqrt{k}}\}$. Whenever
    $\Lambda$ is a closed, convex set in some $B_{\alpha+\delta}(z)$, then
    $r(\Lambda)\leq\alpha+\delta$.
  \end{proposition}
  Of course, the conclusion of Proposition \ref{Proposition9.4} is stronger
  than the second
  hypothesis of Theorem \ref{Theorem9.4}, since the finite intersection of
  balls of radius
  $\alpha+\delta$ is convex and $\alpha+\delta<\rho(X)$.
  \begin{proof}[Proof of Proposition \ref{Proposition9.3}]
We start with
    \begin{claim}
       Let $\sigma=(x_0,\dots,x_{\ell})\in
      U^{\ell+1}_{\alpha}$, and suppose that $p,q\in w_{\alpha}(\sigma)$ and
      that $x$ is on the minimizing geodesic from $p$ to $q$ (but not equal to
      $p$ or $q$). Then $B_{\epsilon}(x)\subset w_{\alpha}(\sigma)$ for some
      $\epsilon>0$.
          \end{claim}
 \begin{proof}[Proof of Claim] For points $r,s,t$ in a strongly convex neighborhood
    in $X$ we define $\angle rst$ to be the angle that the minimizing geodesic
    from $s$ to $r$ makes with the minimizing geodesic from $s$ to $t$. Let
    $\gamma$ be the geodesic from $p$ to $q$, and for fixed $i$ let $\phi$ be
    the geodesic from $x$ to $x_i$. Now, the angles that $\phi$ makes with
    $\gamma$ at $x$ satisfy $\angle pxx_i+\angle x_ixq=\pi$ and therefore one
    of these angles is greater than or equal to $\pi/2$. Assume, without loss
    of generality that $\theta=\angle pxx_i\geq\pi/2$. Let $c=d(x,x_i)$,
    $r=d(p,x)$ and $d=d(p,x_i)\leq\alpha$ (since $p\in
    w_{\alpha}(\sigma)$). Now consider a geodesic triangle in the sphere of
    curvature $k$ with vertices $p'$, $x'$, and $x_i'$ such that
    $$d(p',x')=d(p,x)=r,\quad d(x',x_i')=d(x,x_i)=c\quad\text{and}\quad\angle
    p'x'x_i'=\theta, $$
    and let $d'=d(p',x_i')$. Then, the hypotheses on $\alpha$ imply that the
    Rauch Comparison Theorem (see for example do-Carmo \cite{27}) holds, and we
    can conclude that $d'\leq d$. However, with $\theta\geq\pi/2$, it follows
    that on a sphere, where $p',x',x_i'$ are inside a ball of radius less than
    the strong convexity radius, that $c'<d'$. Therefore we have $c=c'<d'\leq
    d\leq\alpha$ and there is an $\epsilon>0$ such that $y\in B_{\epsilon}(x)$
    implies $d(y,x_i)\leq\alpha$. Taking the smallest $\epsilon>0$ so that
    this is true for each $i=0,\dots,\ell$ finishes the proof of the claim.
  \end{proof}

  \begin{corollary}[Corollary of Claim] For $\sigma\in U^{\ell+1}_{\alpha}$,
    either
    $w_{\alpha}(\sigma)$ consists of a single point, or every point of
    $w_{\alpha}(\sigma)$ is an interior point or the limit of interior points.
  \end{corollary}
  Now suppose that $\sigma_j\in U^{\ell+1}_{\alpha}$ and $\sigma_j\to\sigma$
  in $U^{\ell+1}_{\alpha}$. We must show $w_{\alpha}(\sigma_j)\to
  w_{\alpha}(\sigma)$, that is
  \begin{enumerate}
  \item[(a)] $\sup_{x\in w_{\alpha}(\sigma_j)}d(x,w_{\alpha}(\sigma))\to 0$,
  \item[(b)] $\sup_{x\in w_{\alpha}(\sigma)}d(x,w_{\alpha}(\sigma_j))\to 0$.
\end{enumerate}
In fact (a) holds for any metric space and any $\alpha>0$. Suppose that
(a) was not true. Then there exists a subsequence (still denoted by
$\sigma_j$), and $\eta>0$ such that
$$
\sup_{x\in w_{\alpha}(\sigma_j)}d(x,w_{\alpha}(\sigma))\geq\eta
$$
and therefore there exists $y_j\in w_{\alpha}(\sigma_j)$ with
$d(y_j,w_{\alpha}(\sigma))\geq\eta/2$.  After taking another subsequence, we
can assume $y_j\to y$. But if $\sigma_j=(x^j_0,\dots,x^j_{\ell})$, and
$\sigma=(x_0,\dots,x_{\ell})$, then $d(y_j,x^j_i)\leq\alpha$ which implies
$d(y,x_i)\leq\alpha$ for each $i$ and thus $y\in w_{\alpha}(\sigma)$. But this
is impossible given $d(y_j,w_{\alpha}(\sigma))\geq\eta/2$.

We use the corollary to the Claim to establish (b). First, suppose that
$w_{\alpha}(\sigma)$ consists of a single point $p$. We show that
$d(p,w_{\alpha}(\sigma_j))\to 0$. Let $p_j\in w_{\alpha}(\sigma_j)$ such that
$d(p,p_j)=d(p,w_{\alpha}(\sigma_j))$. If $d(p,p_j)$ does not converge to $0$
then, after
taking a subsequence, we can assume $d(p,p_j)\geq\eta>0$ for some $\eta$. But
after taking a further subsequence, we can also assume $p_j\to y$ for some
$y$. However, as in the argument above it is easy to see that $y\in
w_{\alpha}(\sigma)$ and therefore $y=p$, a contradiction, and so $(b)$ holds
in this case.

Now suppose that every point in $w_{\alpha}(\sigma)$ is either an interior
point, or the limit of interior points. If $(b)$ did not hold, there would be
a subsequence (still denoted by $\sigma_j$) such that
 $$
 \sup_{x\in w_{\alpha}(\sigma)}d(x,w_{\alpha}(\sigma_j))\geq\eta>0
$$
and thus there exists $p_j\in w_{\alpha}(\sigma
)$ such that
$d(p_j,w_{\alpha}(\sigma_j))\geq\eta/2$. After taking another subsequence, we
can assume $p_j\to p$ and $p\in w_{\alpha}(\sigma)$, and, for $j$ sufficiently
large $d(p,w_{\alpha}(\sigma_j))\geq\eta/4$. If $p$ is an interior point of
$w_{\alpha}(\sigma)$ then $d(p,x_i)<\alpha$ for $i=0,\dots,\ell$. But then,
for all $j$ sufficiently large $d(p,x^j_i)\leq\alpha$ for each $i$. But this
implies $p\in w_{\alpha}(\sigma_j)$, a contradiction. If $p$ is not an
interior point, then $p$ is a limit point of interior points $q_m$. But then,
from above, $q_m\in w_{\alpha}(\sigma_{j_m})$ for $j_m$ large which implies
$d(p,w_{\alpha}(\sigma_{j_m}))\to 0$, a contradiction, thus establishing $(b)$
and finishing the proof of Proposition \ref{Proposition9.3}.
\end{proof}

\begin{proof}[Proof of Proposition \ref{Proposition9.4}] Let $\delta$ be such that
  $\alpha+\delta<\min\{\rho(X),\frac{\pi}{2\sqrt{k}}\}$, and let
  $\Lambda$ be any closed convex set in $B_{\alpha+\delta}(z)$. We will show
  $r(\Lambda)\leq\alpha+\delta$. If $z\in\Lambda$, we are done for then
  $\Lambda\subset B_{\alpha+\delta}(z)$ implies $r(\Lambda)\leq\alpha+\delta$
  by Proposition \ref{Proposition9.1}. If $z\notin \Lambda$ let
  $z_0\in\Lambda$ such that
  $d(z,z_0)=d(z,\Lambda)$ (the closest point in $\Lambda$ to $z$). Now let
  $y_0\in\Lambda$ such that $d(z_0,y_0)=\max_{y\in\Lambda}d(z_0,y)$. Let
  $\gamma$ be the minimizing geodesic from $z_0$ to $y_0$, and $\phi$ the
  minimizing geodesic from $z_0$ to $z$. Since $\Lambda$ is convex $\gamma$
  lies on $\Lambda$. If $\theta$ is the angle between $\gamma$ and $\phi$,
  $\theta=\angle zz_0y_0$, then, by the First Variation Formula of Arc Length
  \cite{27}, $\theta\geq\pi/2$; otherwise the distance from $z$ to points on
  $\gamma$ would be initially decreasing. Let $c=d(z,z_0)$, $d=d(z_0,y_0)$ and
  $R=d(z,y_0)$. In the sphere of constant curvature $k$, let $z'$, $z_0'$,
  $y_0'$ be the vertices of a geodesic triangle such that
  $d(z',z_0')=d(z,z_0)=c$, $d(z_0',y_0')=d(z_0,y_0)=d$ and $\angle
  z'z_0'y_0'=\theta$. Let $R'=d(z',y_0')$. Then by the Rauch Comparison
  Theorem, $R'\leq R$. However, it can easily be checked that on the sphere
  of curvature $k$ holds $d'<R'$, since $z'$, $z_0'$ and $y_0'$ are all within
  a strongly convex ball and $\theta\geq\pi/2$. Therefore $d=d'<R'\leq
  R\leq\alpha+\delta$. Thus $\Lambda\subset B_{\alpha+\delta}(z_0)$ with
  $z_0\in\Lambda$, which implies $r(\Lambda)\leq\alpha+\delta$ by Proposition
  \ref{Proposition9.1}. This finishes the proof of Proposition \ref{Proposition9.4}.
\end{proof}

\noindent The proof of Theorem \ref{Theorem9.5} is finished.
\end{proof}

\appendix
\section{An example whose codifferential does not have closed range}\label{Section10}

For convenience, we fix the scale $\alpha=10$; any large enough value is
suitable for our construction.
We consider a compact metric measure space $X$ of the following type:

As metric space, it has three cluster points $x_\infty, y_\infty, z_\infty$
and discrete points $(x_n)_{n\in\naturals}$ ,$(y_n)_{n\in\naturals}$,
$(z_n)_{n\in\naturals}$ converging to $x_\infty$ ,$y_\infty$, $z_\infty$,
respectively.

We set $K_x:=\{x_k\colon k\in\naturals\cup\{\infty\} \}$,
$K_y:=\{y_k\colon k\in\naturals\cup\{\infty\} \}$, and
$K_z:=\{z_k\colon k\in\naturals\cup\{\infty\} \}$. Then $X$ is the
disjoint union of the three ``clusters'' $K_x,K_y,K_z$.

We require:
$$d(x_\infty,y_\infty)=d(y_\infty,z_\infty)=\alpha\quad\text{and}\quad
d(x_\infty,z_\infty)=2\alpha.$$ 

We also require
$$d(x_k,y_n)<\alpha\quad\text{precisely when }n\in\{2k,2k+1,2k+2\},\;
n\in\naturals,\; 
k\in\naturals\cup\{\infty\},$$
$$d(z_k,y_n)<\alpha\quad\text{precisely when }n\in\{2k-1,2k,2k+1\},\;
n\in\naturals,\; 
k\in\naturals\cup\{\infty\}.$$ We finally require that the clusters
$K_x,K_y,K_z$ have diameter $<\alpha$, and that the distance between $K_x$
and $K_y$ as well as between $K_z$ and $K_y$ is $\ge\alpha$.

Such a configuration can easily be found in an infinite dimensional
Banach space such as $l^1(\naturals)$.  For example, in
$l^1(\naturals)$ consider the canonical basis vectors $e_0,e_1,\dots$,
and set
$$x_\infty:=-\alpha e_0,\; y_\infty:=0,\; z_\infty:=\alpha e_0.$$
Define then
\begin{align*}
  x_k &:=-\left(\alpha+\frac{1}{10k}-\frac{1}{2k}-\frac{1}{2k+1}-\frac{1}{2k+2}\right)e_0 +
  \frac{1}{2k}e_{2k}+\frac{1}{2k+1}e_{2k+1}+\frac{1}{2k+2}e_{2k+2},\\ 
  y_k &:=\frac{1}{k} e_k,\\ 
  z_k &:=\left(\alpha+\frac{1}{10k}-\frac{1}{2k-1}-\frac{1}{2k}-\frac{1}{2k+1}\right)e_0 + \frac{1}{2k-1}e_{2k-1}+\frac{1}{2k}e_{2k}+\frac{1}{2k+1}e_{2k+1}.
\end{align*}

\bigskip

We can now give a very precise description of the open
$\alpha$-neighborhood $U_d$ of the diagonal in $X^d$. It contains all 
the tuples whose entries
\begin{itemize}
\item all belong to $K_x\cup\{y_{2k},y_{2k+1},y_{2k+2}\}$ for some
  $k\in\naturals$; or
\item all belong to $K_y\cup\{x_k,x_{k+1},z_{k+1}\}$ for some
  $k\in\naturals$; or
\item all belong to $K_y\cup\{x_k,z_k,z_{k+1}\}$ for some
  $k\in\naturals$; or
\item all belong to $K_z\cup\{y_{2k-1},y_{2k},y_{2k+1}\}$ for some
  $k\in\naturals$.
\end{itemize}

For the closed $\alpha$-neighborhood, one has to add tuples whose entries
all belong to $K_y\cup\{x_\infty\}$ or to $K_y\cup\{z_\infty\}$.

This follows by looking at the possible intersections of $\alpha$-balls
centered at our points.

\medskip

In this topology, every set is a Borel set. We give $x_\infty$, $y_\infty$,
$z_\infty$ measure zero. When considering $L^2$-functions on the $U_d$ we can
therefore ignore all tuples containing one of these points.

We specify  $\mu(x_n):=\mu(z_n):=2^{-n}$ and $\mu(y_n):=2^{-2^n}$; in
this way, the total mass is finite.

We form the $L^2$-Alexander chain complex at scale $\alpha$ and complement it
by 
$C^{-1}:=\reals^3=\reals x\oplus \reals y\oplus \reals z$; the three summands
standing for the three clusters. The differential $c^{-1}\colon C^{-1}\to
L^2(X)$ is defined by $(\alpha,\beta ,\gamma)\mapsto \alpha
\chi_{K_x}+\beta\chi_{K_y}+\gamma \chi_{K_z}$, where $\chi_{K_j}$ denotes the
characteristic function of the cluster $K_j$.

Restriction to functions supported on $K_x^{*+1}$ defines a bounded
surjective cochain map from the $L^2$-Alexander complex at scale
$\alpha$ for $X$ to the one for $K_x$. Note that $\diam( K_x)
<\alpha$, consequently its Alexander complex at scale $\alpha$ is
contractible.

Looking at the long exact sequence associated to a short exact sequence of
Banach cochain complexes,
therefore, the cohomology of $X$ is isomorphic (as topological vector spaces)
to the cohomology of the kernel of this projection, i.e.~to the cohomology of
the Alexander complex of functions vanishing on $K_x^{k+1}$.

This can be done two more times (looking at the kernels of the restrictions to
$K_y$ and $K_z$), so that finally we arrive at the chain complex $C^*$ of
$L^2$-functions on $X^{k+1}$
vanishing at $K_x^{k+1}\cup K_y^{k+1}\cup K_z^{k+1}$.

In particular, $C^{-1}=0$ and $C^0=0$.

We now construct a sequence in $C^1$ whose differentials converge in $C^2$,
but such that the limit point does not lie in the image of $c^1$.

Following the above discussion, the $\alpha$-neighborhood of the diagonal in
$X^2$ contains in particular the 
``one-simplices'' $v_{k}:=(x_k,z_k)$ and $v'_{k}:=(x_k,z_{k+1})$, and their
``inverses'' $\overline{v_{k}}:-(z_k,x_k)$,
$\overline{v'_{k}}:=(z_{k+1},x_k)$.

We define $f_\lambda\in C^1$ with
$f_\lambda(\overline{v_k}):=f_\lambda(\overline{v_k'}):=-f_\lambda(v_k)$,
$f_\lambda(v_k'):=f_\lambda(v_k):=
b_{\lambda,k}:=2^{\lambda k}$ and $f_n(v)=0$ for all other simplices.

Note that for $0<\lambda<1$
\begin{equation*}
  \begin{split}
    \int_{X^2} \abs{f}^2 =& \sum_{k=1}^\infty \abs{f(v_k)}^2\mu(v_k)
    +\abs{f(v_k')}^2\mu(v_k')+\abs{f(\overline{v_k})}^2\mu(v_k)+\abs{f(\overline{v_k'})}^2\mu(v_k')\\
    =& \sum_{k=1}^\infty 2\cdot\left( 2^{2\lambda k}2^{-2k}+ 2^{2\lambda
        k}2^{-2k-1}\right)
  \end{split}
\end{equation*}
which is a finite sum, whereas for $\lambda =1$ the sum is not longer finite.

Let us now consider $g_\lambda:=c^1(f_\lambda)$. It vanishes on all ``$2$-simplices''
(points in $X^2$) except those of the form
\begin{itemize}
\item $d_k:=(x_k,z_k,z_{k+1})$ and more generally
  $d_k^\sigma:=\sigma(x_k,z_k,z_{k+1})$ for $\sigma\in S_3$ a permutation of
  three entries
\item $e_k:=(x_{k-1},z_{k},x_{k})$ or more generally $d_k^\sigma$ as before
\item on degenerate simplices of the form $(x_k,z_k,x_k)$ etc.~$g$ vanishes
  because $f(x_k,z_k)=-f(z_k,x_k)$.
\end{itemize}

We obtain
\begin{equation*}
  g_\lambda(d_k)=-f(v_k')+f(v_k)=0,
  g_\lambda(e_k)=f(\overline{v_k})+f(v_{k-1}')=-2^{\lambda k}+2^{\lambda
    (k-1)}= 2^{\lambda k}\cdot (2^{-\lambda}-1).
\end{equation*}
Similarly, $g_\lambda(d_k^\sigma)=0$ and
$g_\lambda(e_k^\sigma)=\sign(\sigma)g_\lambda(e_k)$.

We claim that $g_1$, defined with these formulas, belongs to $L^2(X^3)$ and is
the limit in $L^2$ of $g_\lambda$ as $\lambda$ tends to $1$.

To see this, we simply compute the $L^2$-norm
\begin{equation*}
  \begin{split}
    \int_{X^3}\abs{g_1-g_\lambda}^2 = &6 \sum_{k=1}^\infty
    \abs{2^{k-1}-2^{\lambda k}(1-2^{-\lambda})}^2 2^{1-3k}\\
    \le & 6 \left(\sup_{k\in \naturals}
      2^{-k/2}\abs{2^{-1}-2^{(\lambda-1)k}(1-2^{-\lambda})}^2\right)\cdot
    \sum_{k=1}^\infty 2^{1-k/2}\\
    &\xrightarrow{\lambda\to 1} 0
  \end{split}
\end{equation*}
(the factor $6$ comes from the six permutations of each non-degenerate simplex
which each contribute equally).

The supremum tends to zero because each individual term does so even without
the factor $2^{-k/2}$ and the sequence is bounded.

\medskip

Now we study which properties an $f\in C^1$ with $c^1(f)=g_1$ has to have.

Observe that for an arbitrary $f\in C^1$, $c^1f(e^\sigma_k)$ is determined
by $f(v_k)$, $f(\overline{v_k})$, $f(v_{k-1}')$, $f(\overline{v_{k-1}'})$ (as
$f$ vanishes on $K_x$).

If $c^1f$ has to vanish on degenerate simplices (and this is the case for
$g_1$), then $f(v_k)=-f(\overline{v_k})$ and $f(v_k')=-f(\overline{v_k'})$.

$c^1f(d^\sigma_k)=0$ then implies that $f(v_k)=f(v_k')$.

It is now immediate from the formula for $c^1 f(d_k)$ and $c^1f(e_k)$ that the
values of $f$ at $v_k$, $v_k'$ are determined by $c^1f(d_k)$, $c^1f(e_k)$ up to
addition of a constant.

Finally, observe that (in the Alexander cochain complex without growth
conditions) $f_1$ (which is not in $L^2$) satisfies $c^1(f_1)=g_1$.

As constant functions are in $L^2$, we observe that $f_1+K$ is not in $L^2$ for any
$K\in\reals$. Nor is any function $f$ on $X^2$ which coincides with $f_1+K$ on
$v_k$, $v_k'$, $\overline{v_k}$, $\overline{v_k'}$.

But these are the only candidates which could satisfy $c^1(f)=g_1$. It follows
that $g_1$ is not in the image of $c^1$. On the other hand, we constructed it
in such a way that it is in the closure of the image. Therefore the image is
not closed.

\subsection{A modified example where volumes of open and closed balls coincide}

The example given has one drawback: although at the chosen scale $\alpha$
open and closed balls coincide in volume (and even as sets, except for
the balls around $x_\infty,y_\infty,z_\infty$) for other balls this is
not the case --- and necessarily so, as we construct a
zero-dimensional object.

We modify our example as follows, by replacing each of the points
$x_k,y_k,z_k$ by a short interval: inside $X\times[0,1]$, with
$l^1$ metric (that is, $d_Y((x,t),(y,u))=d_X(x,y)+|t-u|$), consider
\[Y=\bigcup_{k\in\naturals\cup\{\infty\}}\{x_k,y_k,z_k\}\times[0,1/(12k)].\]

For conveniency, let us write $I_{x,k}$ for the interval
$\{x_k\}\times[0,1/(12k)]$, and similarly for the $y_k$ and $z_k$.  We
then put on each of these intervals the standard Lebesgue measure
weighted by a suitable factor, so that $\mu_Y(I_{x,k})=\mu(x_k)$,
and similarly for the $y_k$ and $z_k$.

Now, if two points $x_k,y_n$ are at distance less than $\alpha$ in $X$,
then they are at distance $<\alpha-1/k$; the corresponding statement holds for
all other 
pairs of points. On the other hand, because of our choice of metric,
$d((x_k,t),(y_n,s))\ge d(x_k,y_n)$ and again the corresponding statement holds
for all other pairs of points in $Y$. It follows that the
$\alpha$-neighborhood of the diagonal 
in $Y^k$ is the union of products of the corresponding intervals, and
exactly those intervals show up where the corresponding tuple is
contained in the $1$-neighborhood of the diagonal in $X^k$.

It is now quite hard to explicitly compute the cohomology of the $L^2$-Alexander
cochain complex at scale $\alpha$.

However, we do have a projection $Y\to X$, namely the projection on
the first coordinate. By the remark about the $\alpha$-neighborhoods, this
projection extends to a map from the $\alpha$-neighborhoods of $Y^k$ onto
those of $X^k$, which is compatible with the projections onto the
factors.

It follows that pullback of functions defines a bounded cochain map
(in the reverse direction)
between the $L^2$-Alexander cochain complexes at scale $\alpha$. Note that
this is an isometric embedding by our choice of the measures.

This cochain map has a one-sided inverse given by integration of a function on
(the $\alpha$-neighborhood of the diagonal in)
$Y^k$ over a product
of intervals (divided by the measure of this product) and assigning this value
to the corresponding tuple in $X^k$. By Cauchy-Schwarz, this is bounded with
norm $1$.

As pullback along projections commutes with the weighted integral we are
using, one checks easily that this local integration map also is a cochain
map for our $L^2$-Alexander complexes at scale $\alpha$.

Consequently, the induced maps in cohomology give an isometric inclusion with
inverse between the cohomology of $X$ and of $Y$.

We have shown that in $H^2(X)$ there are non-zero classes of norm
$0$. Their
images (under pullback) are non-zero classes (because of the retraction given
by the integration map) of norm $0$. Therefore, the cohomology of $Y$ is
non-Hausdorff, and the first codifferential doesn't have closed image, either.

\section[An Example Space with Infinite Dimensional alpha-Scale
Homology]{\boldmath An example of a space with infinite dimensional $\alpha$-scale homology}


The work in in the main body of this paper has inspired the question of the
existence of a  separable, compact metric space with infinite dimensional
$\alpha$-scale homology.  This appendix provides one such example and further
shows the sensitivity of the homology to changes in $\alpha$.

Let $X$ be a separable, complete metric space with metric $d$, and $\alpha > 0$ a ``scale''. Define an associated (generally infinite) simplicial complex $C_{X,\alpha}$ to $(X, d, \alpha)$.  Let $X^{\ell+1}$, for $\ell > 0$, be the $(\ell+1)$-fold Cartesian product, with metric denoted by $d, d : X^{\ell+1} \times X^{\ell+1} \rightarrow \mathbb{R}$ where $d(x; y) = \max_{i=0,\dots,\ell} d(x_i; y_i)$. Let
$$U^{\ell+1}_{\alpha}(X) = U^{\ell+1}_{\alpha} = \{x \in X^{\ell+1} : d(x;D^{\ell+1}) \le \alpha\}$$ where $D^{\ell+1} \subset X^{\ell+1}$ is the diagonal, so $D^{\ell+1} = \{t \in X: (t,\dots,t), \ell + 1 \; \textrm{times}\}$. Let $C_{X,\alpha} = \cup_{t=0}^{\infty} U^{\ell+1}_{\alpha}$.  This has the structure of a simplicial complex whose $\ell$
simplices consist of points of $U^{\ell+1}_{\alpha}$.

The $\alpha$-scale homology is that homology generated by the simplicial complex above.

The original exploration of example compact metric spaces resulted in low
dimensional $\alpha$-scale homology groups.  Missing from the results were any
examples with infinite dimensional homology groups.  In addition examination
of the first $\alpha$-scale homology group was less promising for infinite
dimensional results; the examination resulted in the proof that the first
homology group is always finite, as shown in Section \ref{Section9}.

The infinite dimensional example in this paper was derived through several failed attempts to prove that the $\alpha$-scale homology was finite.  The difficulty that presented itself was the inability to slightly perturb vertices and still have the perturbed object remain a simplex.  This sensitivity is derived from the ``equality'' in the definition of $U^{\ell+1}_{\alpha}$.  It is interesting to note the contrast between the first homology group and higher homology groups.  In the case of first homology group all 1-cycles can be represented by relatively short simplices; there is no equivalent statement for higher homology groups:

\begin{lemma}
A 1-cycle in $\alpha$-scale homology can be represented by simplices with length less than or equal to $\alpha$.
\end{lemma}

\begin{proof}
For any $[x_i, x_j]$ simplex with length greater than $\alpha$ there exists a point $p$ such that $d(x_i, p) \le \alpha$ and $d(x_j, p) \le \alpha$.  This indicates $[x_i, p]$, $[p, x_j]$, and $[x_i, p, x_j]$ are simplices.  Since $[x_i, p, x_j]$ is a simplex we can substitute $[x_i,p]+[p,x_j]$ for $[x_i,x_j]$ and remain in the original equivalence class.
\end{proof}

In the section that follows we present an example that relies on the rigid nature of the definition to produce an infinite dimensional homology group.  The example is a countable set of points in $\mathbb{R}^3$.  One noteworthy point is that from this example it is easy to construct a 1-manifold embedded in $\mathbb{R}^3$ with infinite $\alpha$-scale homology.  In addition to showing that for a fixed $\alpha$ the homology is infinite, we consider the sensitivity of the result around that fixed $\alpha$.

The existence of an infinite dimensional example leads to the following question for future consideration:  are there necessary and sufficient conditions on $(X,d)$ for the $\alpha$-scale homology to be finite.

\subsection {An Infinite Dimensional Example}

The following example illustrates a space such that the second homology group is infinite.  For the discussion below fix $\alpha = 1$.

Consider the set of point $\{A,B, C, D\}$ in the diagram below such
that $$d(A,B)=d(B,C)=d(C,D)=d(A,D) = 1$$ $$d(A,C)=d(B,D)= \sqrt{2}$$ 

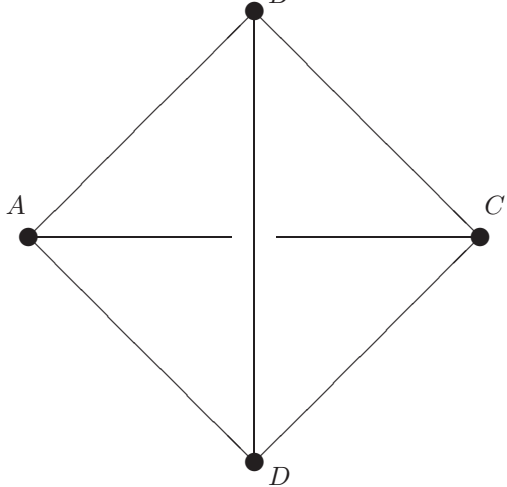
\definecolor{hellgrau}{gray}{.8}
\definecolor{dunkelblau}{rgb}{0, 0, .7}
\definecolor{roetlich}{rgb}{1, .7, .7}
\definecolor{dunkelmagenta}{rgb}{.3, 0, .3}

\begin{center}











\unitlength 0.6mm
  \begin{picture}(100,100)(-50,-50)
    \put (0,-50) {\circle*{4}}
    \put (-50,0) {\circle*{4}}
    \put (0,50) {\circle*{4}}
    \put (50,0) {\circle*{4}}
    \put (-50,0) {\line(1,1){50}}
    \put (-50,0) {\line(1,-1){50}}
    \put (-50,0) {\line(1,0){45}}
    \put (50,0) {\line(-1,0){45}}
    \put (0,-50) {\line(0,1){100}}
    \put (0,-50) {\line(1,1){50}}
    \put (0,50) {\line(1,-1){50}}
    \put (-55,5) {$A$}
    \put (3,52) {$B$}
    \put (51,5) {$C$}
    \put (3,-55) {$D$}
  \end{picture}
\end{center}

The lines in the diagram suggest the correct structure of the $\alpha$-simplices for $\alpha = 1$.  The 1-simplices are $\{\{A,B\}, \{B,C\}, \{C,D\}, \{A,D\}, \{A,C\}, \{B,D\}\}$.  The 2-simplices are the faces $\{\{A,B,C\}, \{A,B,D\}, \{A,C,D\}, \{B,C,D\}\}$.  There are no (non-degenerate) 3-simplices.  A 3-simplex would imply a point such that all of the points are within $\alpha=1$ --- no such point exists.  The chain $[ABC]-[ABD]+[ACD]-[BCD]$ is a cycle with no boundary.

Define $R$ as $R = \{r \in [0, 1, 1/2, 1/3, \dots]\}$. Note in this example $R$ acts as an index set and the dimension of the homology is shown to be at least that of $R$.

Let $X = \{A,B,C,D\} \times R$.  Define $A_r = (A,r)$, $B_r = (B,r)$, $C_r = (C,r)$, and $D_r = (D,r)$.

We can again enumerate the 1-simplices for $X$.  Let $r,s,t,u \in R$.  The 1-simplices are
\begin{gather*}
\{\{A_r,B_s\}, \{B_r,C_s\}, \{C_r,D_s\}, \{A_r,D_s\},\\
 \{A_r,A_s\}, \{B_r,B_s\}, \{C_r,C_s\}, \{D_r,D_s\},\\
\mathbf{\{B_r,D_r\},\{A_r,C_r\}}\}.
\end{gather*}
The last two 1-simplices (highlighted) must have the same index in $R$ due to the distance constraint.

\noindent The 2-simplices are
\begin{align*}
\{
\{A_r,B_s,C_r\},&  \{A_s,B_r,D_r\},  \{A_r,C_r,D_s\}, \{B_r,C_s,D_r\},\\
\{A_r,B_s,B_r\},&  \{B_r,C_s,C_r\},  \{C_r,D_s,D_r\}, \{A_r,D_s,D_r\},\\
\{A_s,A_r,B_s\},&  \{B_s,B_r,C_s\},  \{C_s,C_r,D_s\}, \{A_r,A_r,D_s\},\\
\{A_r,A_s,A_t\},&  \{B_r,B_s,B_t\},  \{C_r,C_s,C_t\}, \{D_r,D_s,D_t\}\}.
\end{align*}

\noindent The 3-simplices are
\begin{align*}
\{
\{A_r,B_s,B_t,C_r\},&  \{A_s,A_t,B_r,D_r\},  \{A_r,C_r,D_s,D_t\}, \{B_r,C_s,C_t,D_r\},\\
\{A_r,B_t,B_s,B_r\},&  \{B_r,C_t,C_s,C_r\},  \{C_r,D_t,D_s,D_r\}, \{A_r,D_t,D_s,D_r\},\\
\{A_t,A_s,A_r,B_s\},&  \{B_t,B_s,B_r,C_s\},  \{C_t,C_s,C_r,D_s\}, \{A_t,A_r,A_r,D_s\},\\
\{A_r,A_s,A_t,A_u\},&  \{B_r,B_s,B_t,B_u\},  \{C_r,C_s,C_t,C_u\}, \{D_r,D_s,D_t,D_u\}\}.
\end{align*}

Define $\sigma_r := [A_r B_r C_r]-[A_r B_r D_r]+[A_r C_r D_r]-[B_r C_r D_r]$.  By calculation, $\sigma_r$ is shown to be a cycle.  Suppose that there existed a chain of 3-simplices such that the $\sigma_r$ is the boundary then $\gamma = [A_r A_s B_r D_r]$ must be included in the chain to eliminate $[A_r B_r D_r]$.  Since the boundary of $\gamma$ contains $[A_s B_r D_r]$ there must be a term to eliminate this term.  The only term with such a boundary is of the form $[A_s A_t B_r D_r]$.  Again, a new simplex to eliminate the extra boundary term is in the same form.  Either this goes on \emph{ad infinitum}, impossible since the chain is finite, or it returns to $A_r$ in which case the boundary of the original chain is 0 (contradicting that the $[A_r B_r D_r]$ term is eliminated).  For all $r\in R$, $\sigma_r$ is a generator for homology.

If $s \neq t$ then $\sigma_s$ and $\sigma_t$ are not in the same equivalence class.  Suppose they are.  The same argument above shows that any term with the face $[A_t B_t D_t]$ will necessarily have a face $[A_u B_t D_t]$ for some $u\in R$.  Such a term needs to be eliminated since it cannot be in the final image but such an elimination would cause another such term or cancel out the $[A_t B_t D_t]$.  In either case the chain would not satisfy the boundary condition necessary to equivalence  $\sigma_s$ and $\sigma_t$ together.

Each  $\sigma_s$ is a generator of homology and, therefore, the dimension of the homology is at least the cardinality of $R$ which in this case is infinite.
\begin{theorem}
  For $\alpha = 1$, the second $\alpha$-scale homology group for $$X =
  \{A,B,C,D\} \times R$$ is infinite dimensional.
\end{theorem}

\subsection[Consideration for alpha<1]{\boldmath Consideration for $\alpha < 1$}
The example above is tailored for scale $\alpha = 1$.  In this metric space the nature of the second $\alpha$-scale homology group changes significantly as $\alpha$ changes.

Consider when $\alpha$ falls below one.  In this case the structure of the simplices collapses to simplices restricted to a line (with simplices of the form $\{\{A_r,A_s,A_t\},  \{B_r,B_s,B_t\},  \{C_r,C_s,C_t\}, \{D_r,D_s,D_t\}\}$).  These are clearly degenerate simplices resulting in a trivial second homology group.

In this example the homology was significantly reduced as $\alpha$ decreased.  This is not necessarily always the case.  The above example could be further enhanced by replicating smaller versions of itself in a fractal-like manor so that as $\alpha$ decreases the $\alpha$-scale homology encounters many values with infinite dimensional homology.

\subsection[Consideration for alpha>1]{\boldmath Consideration for $\alpha > 1$}
There are two cases to consider when $\alpha >1$.  The first is the behavior for very large $\alpha$ values.  In this case the problem becomes simple as illustrated by the lemma below.

Define $\mathbf{\alpha}$ \textbf{large} with respect to $d$ if there exists an $\rho \in X$ such that $d(\rho,y) \le \alpha$ for all $y \in X$.

\begin{lemma}
Let $X$ be a separable, compact metric space with metric $d$.  If $\alpha$ is large with respect to $d$ then the $\alpha$-scale homology of $X$ is trivial.
\end{lemma}
\begin{proof}
Let $\rho \in X$ satisfy the attribute above.  Then $U^{\ell+1}_{\alpha} = X^{\ell+1}$ since $$d((\rho,\dots,\rho), (x_0, x_1,\dots,x_\ell)) \le \alpha$$ for all values of $x_i$.

Let $\sigma = \sum_{j=1}^{k} c_j (a_0^j, a_1^j,\dots,a_n^j)$ be an $n$-cycle.    Define $$\sigma_\rho = \sum_{j=1,k} c_j (a_0^j, a_1^j,\dots,a_n^j, \rho).$$ The $n+1$-cycle, $\sigma_\rho$, acts as a cone with base $\sigma$.  Since $\sigma$ is a cycle the terms in the boundary of $\sigma_\rho$ containing $\rho$ cancel each other out to produce zero.  The terms without $\rho$ are exactly the original $\sigma$.  Therefore there exists no cycles without boundaries.  This proves that for $\alpha$ large and $X$ infinite the homology of $X$ is trivial.
\end{proof}

In the case that $\alpha > 1$ but is still close to $1$, the second homology
group changes significantly but does not completely disappear.  In the
example, simplices that existed only by the equality in the definition of
$\alpha$-scale homology when $\alpha =1$ now find neighboring $2$-simplices
joined by higher dimensional $3$-simplices.  The result is larger equivalence
classes of $2$-cycles.  This reduces the infinite dimensional homology for
$\alpha = 1$ to a finite dimension for $\alpha$ slightly larger than $1$. 
As $\alpha$ gets closer to $1$ from above the dimension of the homology increases without bound.

It is interesting to note that the infinite characteristics for $\alpha = 1$ are tied heavily to the fact that the simplices that determined the structure lived on the bounds of being simplices.  As $\alpha$ changes from $1$, the rigid restrictions on the simplices is no longer present in this example.  The result is a significant reduction in the dimension of the homology.

\section{Open problems and remarks}

Throughout the text, we have attempted to give indications to promising areas of
future research. Here we summarize some of the main points.

\begin{itemize}
\item  How do the methods of this paper apply to concrete examples, in
  particular, to 
the data in Carlsson et al. \cite{3}? Specifically, can we recognize surfaces?
Which 
substitutes for torsion do we have at hand?
\item  For non-oriented manifolds, can one introduce a twisted version of the
  coefficients 
that would make the top-dimensional Hodge cohomology visible?
\item  Is the Hodge cohomology independent of the choice of neighborhoods
  (Vietoris-Rips or ours)? Under which properties of metric spaces are the
  images of the corestriction maps (mentioned in Remark \ref{remFunctorial}
  independent of these choices?
\item The Cohomology Identification Problem (Question \ref{qCIP}: To what
  extent are $H^{\ell}_{L^2 ,\alpha} (X) $ 
and $H^{\ell}_\alpha (X)$ isomorphic?
\item The Continuous Hodge Decomposition (Question \ref{qCHD}: Under what
  conditions on $X$ 
and $\alpha > 0$ is it true that there is an orthogonal direct sum
decomposition of $C^{\ell+1}_\alpha$ 
in boundaries, coboundaries, and harmonic functions?
\item The Poisson Regularity Problem (Question \ref{qPRP}: For $\alpha > 0$ and $\ell> 0$, suppose that
$\Delta f = g$ where $g \in C^{\ell+1}_\alpha$ and $f \in L^2_\alpha (U_α^{\ell +1} )$. Under what conditions on $(X, d, \mu)$
is $f$ continuous?
\item 
The Harmonic Regularity Problem (Question \ref{qHRP}: {qHRP} For $\alpha>0$, and $\ell>0$,
  suppose that $\Delta f=0$ where $f\in L^2_a(U^{\ell+1}_{\alpha})$. What
  conditions on $(X,d,\mu)$ would imply $f$ is continuous?

\end{itemize}

\bibliographystyle{plain}
\bibliography{smale}

\end{document}